\documentclass[11pt]{article}
\usepackage{latexsym,amsfonts,amssymb,amsmath,amsthm}
\usepackage{graphicx}
\usepackage{cite}
\usepackage[usenames,dvipsnames]{color}
\usepackage{ulem}
\usepackage[colorlinks=true]{hyperref}
\hypersetup{urlcolor=blue, citecolor=red}

\newtheorem{theorem}{Theorem}[section]
\newtheorem{definition}[theorem]{Definition}

\newtheorem{lemma}[theorem]{Lemma}
\newtheorem{rem}[theorem]{Remark}
\newtheorem{prop}[theorem]{Proposition}
\newtheorem{cor}[theorem]{Corollary}
\numberwithin{equation}{section}

  \def \G{\Gamma}

\makeatletter 
\@addtoreset{equation}{section}
\makeatother  

\newcommand\norm[1]{\lVert#1\rVert}
\newcommand\abs[1]{\lvert#1\rvert}

\newcommand\RR{\ensuremath{\mathbb{R}}}
\newcommand\eps{\varepsilon}
\def\tint{\text{Int}}
\newcommand{\ucite}[1]{\cite{#1}}

\begin{document}
\title{Generical behavior of flows strongly monotone with respect to high-rank cones}

\author{Lirui Feng\thanks{Supported by the NSERC and the NSERC-IRC Program.}\\
Department of Mathematics and Statistics\\
 York University, Toronto, ON, M3J 1P3, Canada
 \\
 \\
Yi Wang\thanks{Partially supported by NSF of China No.11825106, 11771414, Wu Wen-Tsun Key Laboratory and the Fundamental
Research Funds for the Central Universities.} \\
School of Mathematical Science\\
 University of Science and Technology of China
\\ Hefei, Anhui, 230026, P. R. China
\\
\\
Jianhong Wu\thanks{ Supported by the Canada Research Chairs and the NSERC-IRC Program and the Natural Science and Engineering Research Council of Canada 105588-2011.}\\
 Department of Mathematics and Statistics
 \\ York University, Toronto, ON, M3J 1P3, Canada
  \\
}

\date{}
\maketitle
\indent {\bf Abstract:} We consider a $C^{1,\alpha}$ smooth flow in $\mathbb{R}^n$ which is ``strongly monotone" with respect to a cone $C$ of rank $k$, a closed set that contains a linear subspace of dimension $k$ and no linear subspaces of higher dimension. We prove that orbits with initial data from an open and dense subset of the phase space are either pseudo-ordered or convergent to equilibria. This covers the celebrated Hirsch's Generic Convergence Theorem in the case $k=1$, yields a generic Poincar\'{e}-Bendixson Theorem for the case $k=2$, and holds true with arbitrary dimension $k$. Our approach involves the ergodic argument using the $k$-exponential separation and the associated $k$-Lyapunov exponent (that reduces to the first Lyapunov exponent if $k=1$.)

\section{Introduction}

We are interested in the global dynamics of a semiflow $\Phi_t$ ``monotone with respect to a cone" $C$ of rank-$k$ on a Banach space $X$. Here, a {\it cone $C$ of rank $k$} is a closed subset of $X$ that contains a linear subspace of dimension $k$ and no linear subspaces of higher dimension. It is not a cone defined normally in the literature, but we adopt this definition in honor of the pioneering work of Fusco and Oliva \cite{FO1} in a finite-dimensional space, and of Krasnoselskii et al. \cite{KLS} in a Banach space $X$.

A convex cone $K$ (defined in the nomal sense) does give rise to a cone of rank $1$, $K\cup(-K)$. Therefore, the class of semiflows we consider includes the order-preserving (monotone) semiflows intensively studied since the series of groundbreaking work of Hirsch \ucite{Hir1,Hir2,Hir3,Hir4,Hir5,Hir6} and Matano \cite{Mata1,Mata2}; see also monographs or surveys \cite{Hir-Smi,Smi95,Po1,Smi17} for more details. There is however an essential difference between a convex cone $K$ and a high-rank cone due to the lack of convexity in the latter. The lack of convexity requires substantially new techniques for exploring the implication of monotonicity for the global and/or generic dynamics of the considered semiflows.

An important example of a semiflow monotone with respect to a cone of high-rank is the monotone cyclic feedback system (for example, \cite{M-PSmi90,M-PSe96-2,Ge}) arising from a wide range of cellular, neural and physiological control systems. A special class of a monotone cyclic feedback system with positive feedback generals a semiflow that is monotone in the sense of Hirsch (see, for example, \cite{M-PSmi90}). However, a large class of monotone cyclic feedback systems are with negative feedback, such systems generate semiflows which are monotone with respect to a sequence of nested cones of even-ranks (see, e.g. \cite{LW1,Te1}). Negative feedbacks embedded in a cyclic architecture are believed to be the underling source of oscillatory dynamic patterns. Indeed, a Poincar\'{e}-Bendixson theorem has been established for many monotone cyclic feedback systems with negative feedback even when time lags are involved in the feedback \cite{M-PSe96,M-PSe96-2,M-PN}. Other models arising from important applications, to which our generic convergence theory can be applied, include high dimensional competitive systems (see, for example, \cite{Ba,Hir3,Hir-Smi,Mi1,JMW}); systems with quadratic cones and Lyapunov-like functions (see, for example, \cite{SmithR-2,SmithR-3,OSan}); as well as systems with integer-valued Lyapunov functionals such as scalar parabolic equations on an circle (see, for example, \cite{Fied89,FM89,JR,SWZ,SanF,Te1}) and tridiagonal systems (see, for example, \cite{FGW13,MSon,Smillie,Smith91}).

For the sake of easy reference, in what follows, we call a semiflow, that is monotone with respect to a convex cone $K$ (that induces a 1-cone $K\cup-K$), a {\it classical} monotone semiflow. The core to the huge success of developing the theory and applications of global dynamics for a classical monotone semiflow is the Generic Convergence Theorem, due to Hirsch. The Hirsch's generic convergence theorem concludes that the set of all $x\in X$, for which the omega-limit set $\omega(x)$ of $x$ satisfies $\omega(x)\subset E$ (the set of equilibria), is generic (open-dense, residual) in $X$. Subsequent studies further establish that for a classical smooth strongly monotone systems, precompact semi-orbits are generically convergent to equilibria in the continuous-time case \cite{Smi-Thi1,Pola1,Pola2} or to cycles in the discrete-time case  \cite{PT1,HP,Te2}.

The Generic Convergence Theorem of Hirsch and its extensions are based on an observation that for a classical strongly monotone system, there are exactly two different kinds of nontrival orbits: {\it pseudo-ordered orbits} and {\it unordered orbits}. Here, a nontrivial orbit:  pseudo-ordered orbits and unordered orbits. Here, a nontrivial orbit $O(x):=\{\Phi_t(x) :t \ge 0\}$ is pseudo-ordered if $O(x)$ possesses one pair of distinct ordered-points $\Phi_{t}(x)$ and $\Phi_{s}(x)$ (i.e.,  $\Phi_{t}(x)-\Phi_{s}(x)\in K\cup (-K)$); otherwise, it is called unordered (see Definition \ref{D:two-type-orbit} with $C=K\cup (-K)$).

The fundamental result in the classical strongly monotone semiflow theory, the {\it Monotone Convergence Criterion}, asserts that every pseudo-ordered precompact orbit converges monotonically to equilibria (see, for example, \cite[Theorem 1.2.1]{Smi95}). Further developments from this Monotone Convergence Criterion including the Nonordering of Limit Sets and Limit Set Dichotomy (See, for example, \cite{Hir0,Smi95,Hir-Smi}), provide the critical technical tools used to establish the Generic Convergence Theorem.

In comparison, the structure of an $\omega$-limit set of a pseudo-ordered semi-orbit for a semiflow strongly monotone with respect to a $k(\ge 2)$-cone $C$ is much more complicated, due to the lack of convexity. Sanchez \cite{San09,San10} addressed this problem and showed that the closure of any orbit in the $\omega$-limit set of a pseudo-ordered orbit is ordered with respect to $C$. In particular, Sanchez \cite{San09} obtained a Poincar\'{e}-Bendixson type result for pseudo-ordered orbits when $k=2$: the $\omega$-limit set of a pseudo-ordered orbit containing no equilibria is a closed orbit. The work \cite{San09,San10} for smooth finite dimensional flows, for which the $C^1$-Closing Lemma was required, was recently extended in Feng, Wang and Wu \cite{F-W-W} to a semiflow on a Banach space without the smoothness requirement. It was showed in \cite{F-W-W} that the $\omega$-limit set $\omega(x)$ of a pseudo-ordered semi-orbit admits a trichotomy, i.e., $\omega(x)$ is either ordered; or $\omega(x)\subset E$; or otherwise, $\omega(x)$ possesses a certain ordered homoclinic property. Given the possibility of an ordered homoclinic structure in the $\omega$-limit set of an pseudo-ordered semi-orbit, we anticipate a number of difficulties one has to face to establish an extension of the Hirsch's Generic Convergence Theorem for flows monotone with respect to $k (\geq 2)$-cone $C$.

With $k>1$, due to the loss of convexity, the ``order"-relation defined by $C$ (see Section 2) is not a partial order since it is neither antisymmetric nor transitive. In addition, kernel tools such as the Nonordering of Limit Sets and Limit Set Dichotomy which have played a crucial role in proving Hirsch's Generic Convergence Theorem are no longer valid. We need novel techniques to understand the generic dynamics without these technical tools and under the assumptions that the $\omega$-limit set of even a pseudo-ordered orbit may contain an ordered homoclinic structure.

In the present paper, we focus on the generic behavior of a flow in $\RR^n$ strongly monotone
with respect to $k(\ge 2)$-cone C. Before describing our approach and main results, we formulate the basic assumptions:

\begin{itemize}
\item[{\bf (FWW)}] The flow $\Phi_t$ on $\RR^n$ is strongly monotone with respect to a $k$-cone $C$, is $C^{1,\alpha}$-smooth and its $x$-derivative $D_x\Phi_t$ satisfies that $D_x\Phi_t(C\setminus \{0\})\subset {\rm Int}C$ for $t>0$.
\end{itemize}

We refer to Section 2 for more detailed definitions about strong monotonicity and $H\ddot{o}lder$ continuity. In what follows, we assume (FWW) assumption holds. Let $$Q=\{x\in \mathbb{R}^n: \text{the orbit }O(x) \text{ is pseudo-ordered}\}.$$

\vskip 1mm
\noindent \textbf{Theorem A. (Generic Dynamics Theorem)} Let $\mathcal{D}\subset \RR^n$ be an open, bounded, and positively invariant set. Then the set $\{x\in \mathcal{D}: x\in Q \text{ or }\omega(x) \text{ is a singleton}\}$ contains an open and dense subset of $\mathcal{D}$.
\vskip 2mm

This theorem, in a slightly stronger version, will be proved in Section 5 (Theorem \ref{T:Open-dense}). It concludes that, in finite-dimensional cases, {\it generic orbits} of a smooth flow strongly monotone with respect to $C$ are either pseudo-ordered or convergent to a single equilibrium. Needlessly to say that, if the rank $k=1$, Theorem A in conjunction with the Monotone Convergence Criterion implies the celebrated Hirsch's Generic Convergence Theorem.

In the special case where $k=2$, together with the results obtained in our study \cite{F-W-W}, we have the following:

\vskip 1mm
\noindent \textbf{Theorem B. (Generic Poincar$\mathbf{\acute{e}}$-Bendixson Theorem)} Let $k=2$. Let $\mathcal{D}\subset \RR^n$ be an open, bounded, and positively invariant set. Then, for generic (open and dense) points $x\in D$, the $\omega$-limit set $\omega(x)$ containing no equilibria is a single closed orbit.
\vskip 2mm

Theorem B (and its slightly stronger version Theorem \ref{T:Open-dense-PB}) is titled as the {\it Generic Poincar\'{e}-Bendixson Theorem}. Note that while the Poincar\'{e}-Bendixson type result was obtained in \cite{San09,F-W-W} for just certain (i.e., pseudo-ordered) orbits, here we conclude the Poincar$\acute{e}$-Bendixson property for generic orbits.

\vskip 2mm

A critical ingredient of our approach is the smooth ergodic argument motivated by Pol$\acute{a}\check{c}$ik and Tere$\check{s}\check{c}\acute{a}$k \cite{Po1}. According to Tere\v{s}\v{c}\'ak \cite{Te1}, the cone invariance condition of $D_x\Phi_t$ in {\bf (FWW)} implies that the linear skew-product flow $(\Phi_t,D\Phi_t)$ admits a {\it $k$-exponential separation} along any compact invariant set $\Omega$ associated to the $k$-cone $C$.
Roughly speaking, this property describes that the product bundle $\Omega\times \RR^n$ admits a decomposition $\Omega\times \RR^n=E\bigoplus F$ into two closed invariant subbundles of $(\Phi_t,D\Phi_t)$, one with $k$-dimensional fibres $\{E_x\}_{x\in \Omega}$, the other with $(n-k)$-dimensional fibres $\{F_x\}_{x\in \Omega}$ not containing a nonzero vector in $C$, such that for any $(x,v)\in \Omega\times \RR^n$ with $v\notin F_x$, the $E$-component of $D_x\Phi_tv$ dominates the $F$-component as $t\to +\infty$ (see Definition \ref{D:ES-separation} or its versions for random dynamics in \cite{LW1,LW2}). This $k$-exponentially separated continuous decomposition is
also called dominated-splitting in Differentiable Dynamical Systems (see, for example, \cite{Puj,BG,BDV,QTZ}) and in control theory (see, for example, \cite{CK}).

When $C$ is a $1$-cone and $\Omega$ is just a single point, the $k$-exponential separation is equivalent to the celebrated Krein-Rutman Theorem \cite{KR,M-PN1,Nu} (or Perron-Frobenius Theorem \cite{Gan} in finite-dimensions). While, for strongly positive linear skew-product (semi)flows or random dynamical systems, the existence of $1$-exponential separation associated to a convex cone $K$ can be found in \cite{HuPo,Mi2,MiSh1,MiSh2,MiSh3,MiSh4,PT2,Ru} with important applications. In particular, the $1$-exponential separation plays a key role in proving the generic convergence to cycles for classical smooth strongly monotone discrete-time dynamical systems \cite{PT1,Te2}.

The $k$-exponential separation is one of the major tools in the proof of our main results. In particular, the $k$-exponential separation allows us to employ the linearization to examine the behavior of orbits of $\Phi_t$ near a compact invariant set $\Omega$. Another critical tool we use is {\it the $k$-Lyapunov exponent} of $x\in \Omega$, defined as
\begin{equation*}
\lambda_{kx}=\limsup_{t\to +\infty}\dfrac{\log m(D_x\Phi_t|_{_{E_x}})}{t},
\end{equation*}
where $m(D_x\Phi_t|_{_{E_x}})=\inf\limits_{v\in E_x\cap S}\norm{D_x\Phi_tv}$ is
the infimum norm of $D_x\Phi_t$ restricted to $E_x$. Properties of $k$-exponential separation and
$\lambda_{kx}$ will be described in Section 3, among which is an important characterization of a point in $k$-cone $C$ in terms of its projections to $E$ and $F$ of the $k$-exponential separation (Lemma \ref{L:ES-Vs-Cone}).

The Multiplicative Ergodic Theorem (see, for example, \cite{LL,OSe,Vi,John-Pa-S}), ensures that $\lambda_{kx}$ is actually the limit, not just the superior limit, for ``most" points $x\in \Omega$. Such points for which $\lambda_{kx}$ is the limit are said to be regular; and other points are said to be irregular. According to the sign of the $k$-Lyapunov exponents of the regular/irregular points on any given $\omega$-limit set $\omega(x)$, we develop our technical proofs for main Theorems into three cases in Section 4. Firstly, we show that if $\omega(x)$ contains a regular point $z$ such that $\lambda_{kz}\le 0$, then either $x\in Q$ or $\omega(x)$ is a singleton (Theorem \ref{omega-lambdak<=0}). Secondly, we prove that if $\lambda_{kz}>0$ for any $z\in \omega(x)$, then $x$ is highly unstable (Lemma \ref{Distance-nonlinear-system}), and belongs to the closure $\overline{Q}$ (Theorem \ref{omega-lambdak>0}). Thirdly, we show that if $\lambda_{k\tilde{z}}>0$ for any regular point $\tilde{z}\in \omega(x)$ but there is an irregular point $z$ with $\lambda_{kz}\leq 0$, then $x\in \overline{Q}$ (Theorem \ref{regular-lambdak>0}).

The smooth ergodic argument of Pol\'{a}\v{c}ik and Tere\v{s}\v{c}\'ak \cite{PT1} was used in our proof. However, for $k(\ge 2)$-cones, the loss of convexity of $C$ prevents the utilization of the order-topology (that is valid only for $1$-cone) to estimate the evolution of the orbits near $\omega(x)$. We overcome such a difficulty by Lemma \ref{L:ES-Vs-Cone} that provides a characterization of the point in $C$ in terms of its projections to the bundles $E$ and $F$ associated with the $k$-exponential separation. We should mention that for the flows monotone with respect to a 1-cone, some ergodic theory arguments were used in \cite{PT1} to exclude the possibility of the third case. This case may happen when $k>1$, so we need to deal with this case by estimating the $t$-derivative of $\Phi_t$ at the irregular points.

Our generic convergence and generic Poincar\'{e}-Bendixson theorems hold for any flow satisfying (FWW). We illustrate these general results in Section 6 by an application to an $n$-dimensional ODE system with a quadratic cone.  R. A. Smith \cite{SmithR-2,SmithR-3} proved a Poincar\'{e}-Bendixson theorem for this system by assuming the existence of a certain infinite quadratic Lyapunov function, and Sanchez \cite{San09} established the connections between Smith's results and the smooth systems strongly monotone with respect to some cones of rank-$2$. Our general theory allows us to establish a {\it generic Poincar\'{e}-Bendixson Theorem} for such a high-dimensional system even if a quadratic Lyapunov function can not be constructed.

\section{Notations and Preliminary Results}

\indent We start with some notations and a few definitions. Let $\mathbb{R}^n$ be the $n$-dimensional Euclidean space with a norm $\norm{\cdot}$. A flow on $\RR^n$ is a continuous map $\Phi:\mathbb{R}\times \RR^n\to \RR^n$ with: $\Phi_0={\rm Id}$ and $\Phi_t\circ\Phi_s=\Phi_{t+s}$ for $t,s\in \RR$. Here, $\Phi_t(x)=\Phi(t,x)$ for $t\in \RR$ and $x\in \RR^n$ and ${\rm Id}$ is the identity map on $\RR^n$.
 A flow $\Phi_t$ on $\RR^n$ is {\it $C^{1,\alpha}$-smooth} if $\Phi|_{\RR\times \RR^n}$ is a $C^{1,\alpha}$-map (a $C^1$-map with a locally $\alpha$-H\"{o}lder derivative) with $\alpha\in (0,1]$. The derivatives of $\Phi_t$ with respect to $x$, at $(t,x)$, is denoted by $D_x\Phi_t(x)$.

Let $x\in \mathbb{R}^n$,  the {\it orbit of $x$} is denoted by $O(x)=\{\Phi_t(x):t\ge 0\}$.
An {\it equilibrium} (also called {\it a trivial orbit}) is a point $x$ for which $O(x)=\{x\}$. Let $E$ be the set of all equilibria of $\Phi_t$. A nontrivial orbit $O(x)$ is said to be a {\it$T$-periodic orbit} for some $T>0$ if $\Phi_T(x)=x$.  The {\it$\omega$-limit set} $\omega(x)$ of $x\in \RR^n$ is defined by $\omega(x)=\cap_{s\ge 0}\overline{\cup_{t\ge s}\Phi_t(x)}$. If $O(x)$ is bounded, then $\omega(x)$ is nonempty, compact, connected and invariant.

 A subset $D$ is called {\it positively invariant} if $\Phi_{t}(D)\subset D$ for any $t\geq 0$, and is called {\it invariant} if $\Phi_{t}(D)=D$ for any $t\in \mathbb{R}$. Given a subset $D\subset \RR^n$, {\it the orbit $O(D)$ of $D$} is defined as $O(D)=\bigcup_{x\in D}O(x)$.
 A subset $D$ is called {\it $\omega$-compact} if $O(x)$ is bounded for each $x\in D$ and $\bigcup_{x\in D}\omega(x)$ is bounded. Clearly, $D$ is $\omega$-compact provided that the orbit $O(D)$ is bounded.

\vskip 2mm
Now we define $k$-cones of $\mathbb{R}^n$.
\begin{definition}\label{D:k-cone}
 A closed set $C\subset \mathbb{R}^n$ is called a cone of rank-$k$ ({\it abbr. $k$-cone}) if the following
 are satisfied:

 {\rm i)} For any $v\in C$ and $l\in \RR,$ $lv\in C$;

 {\rm ii)}  $\max\{\dim W:C\supset W \text{ linear subspace}\}=k.$

 \noindent Moreover, the integer $k(\ge 1)$ is called the rank of $C$.
\end{definition}
A $k$-cone $C\subset \RR^n$ is said to be {\it solid} if its interior $\tint C\ne \emptyset$; and $C$ is called {\it $k$-solid} if there is a $k$-dimensional linear subspace $W$ such that $W\setminus \{0\}\subset \tint C$. Given a $k$-cone $C\subset \RR^n$, we say that $C$ is {\it complemented} if there exists a $k$-codimensional space $H^{c}\subset \RR^n$ such that $H^{c}\cap C=\{0\}$.

For two points $x,y\in \RR^n$, we say that {\it $x$ and $y$ are ordered, denoted by $x\thicksim y$}, if $x-y\in C$. Otherwise, $x,y$ are called to be {\it unordered}. The pair $x,y\in X$ are said to be {\it strongly ordered}, denoted by $x\thickapprox y$, if $x-y\in \tint C$.

\begin{definition}\label{D:stongly-monotne}
 A flow $\Phi_t$ on $\RR^n$ is called {\it monotone with respect to a $k$-solid cone $C$} if
$$\Phi_t(x)\thicksim\Phi_t(y),\, \text{ whenever } x\thicksim y \text{ and } t\ge 0;$$ and $\Phi_t$ is called {\it strongly monotone with respect to $C$} if $\Phi_t$ is monotone with respect to $C$ and
$$\Phi_t(x)\approx \Phi_t(y),\, \text{ whenever } x\ne y, x\thicksim y \text{ and } t>0.$$
\end{definition}

Throughout this paper, we always impose the following assumptions:

\begin{itemize}
\item[{\bf (FWW)}] The flow $\Phi_t$ is $C^{1,\alpha}$-smooth and strongly monotone with respect to a $k$-cone $C$. Moreover, the $x$-derivative $D_x\Phi_t$ satisfies $D_x\Phi_t(C\setminus \{0\})\subset {\rm Int}C$ for $t>0$.
\end{itemize}

\noindent We note that, due to the lack of convexity, the cone invariance condition in (FWW) does not have to imply $\Phi_t$ is strongly monotone with respect to $k(\ge 2)$-cone $C$.

\begin{definition}\label{D:two-type-orbit}
A nontrivial orbit $O(x)$ is called {\it pseudo-ordered} (also called {\it of Type-I}), if there exist two distinct points $\Phi_{t_1}(x),\Phi_{t_2}(x)$ in $O(x)$ such that $\Phi_{t_1}(x)\thicksim\Phi_{t_2}(x)$. Otherwise,
 $O(x)$ is called {\it unordered} (also called {\it of Type-II}).
\end{definition}
Hereafter, we let
$Q=\{x\in \mathbb{R}^n: \text{the orbit }O(x) \text{ is pseudo-ordered}\}.$ Clearly, $Q$ is an open subset of $\RR^n$ provided that $\Phi_t$ is strongly monotone with respect to $C$.
The following co-limit Lemma, which will be useful in our proof, is due to \cite[Lemma 4.3]{F-W-W}.

\begin{lemma}\label{co-limit}
Assume that $\Phi_t$ is strongly monotone with respect to $C$. If $x_1\thicksim x_2$ and there is a sequence $t_n\to \infty$ such that $\Phi_{t_k}(x_1)\rightarrow z$ and $\Phi_{t_k}(x_2)\rightarrow z$, then one has $z\in Q\cup E$.
\end{lemma}

\section{$k$-Exponential Separation and $k$-Lyapunov Exponents}

In the first part of this section, we will focus on a crucial tool in our approach called the $k$-exponential separation
 along any compact invariant set associated to the $k$-cone $C$.

 Let $G(k,\mathbb{R}^n)$ be {\it the Grassmanian of $k$-dimensional linear subspace of $\mathbb{R}^n$}, which consists of all $k$-dimensional linear subspace in $\mathbb{R}^n$. $G(k, \mathbb{R}^n)$ is a completed metric space by endowing {\it the gap metric} (see, for example, \cite{Kato,LL}). More precisely, for any nontrivial closed subspaces $L_1,L_2\subset \RR^n$, define that
 \begin{equation*}\label{E:GapDistance}
 d(L_1,L_2)=\max\left\{\sup_{v\in E\cap S}\inf_{u\in F\cap S}|v-u|, \sup_{v\in F\cap S}\inf_{u\in E\cap S}|v-u|\right\},
 \end{equation*}
 where $S=\{v\in \RR^n:\norm{v}=1\}$ is the unit ball.
 For a solid $k$-cone $C\subset X$, we denote by $\G_k(C)$ the set of $k$-dimensional subspaces inside $C$, i.e., $$\G_k(C)=\{L\in G(k,\mathbb{R}^n):\,L\subset C\}.$$

 Let $K\subset \mathbb{R}^n$ be a compact subset and $S_t$ be a flow on $K$. Let $\{R_x^t,\,x\in K,t\in \RR\}$ be a family of bounded linear maps on $\RR^n$ satisfying: (i). the map
  $(t,x)\mapsto R_x^t: K\times \RR\to L(\RR^n)$ is continuous; (ii). $R_x^{t_1+t_2}=R_{S_{t_2}x}^{t_1}\circ R_x^{t_2}$ for all $t_1,t_2\in \RR$ and $x\in K$. The pair of families of maps $(S_t,R^t)$ is called a vector bundle (or linear skew-product) flow on $K\times \RR^n$.
In particular,  $(\Phi_t,\,D\Phi_t)$ is a linear skew-product flow on $K\times \RR^n$.

 Let $\{Y_x\}_{x\in K}$ be a family of $k$-dimensional subspaces of $\mathbb{R}^n$. We call $K\times (Y_x)$ {\it a $k$-dimensional continuous vector bundle} if the map $K \rightarrow G(k,\,\mathbb{R}^n);x\mapsto Y_x$ is continuous. The continuous bundle $K\times (Y_x)$ is invariant with respect to $(\Phi_t,\,D\Phi_t)$ if $D\Phi_t(x)Y_x= Y_{Fx}$ for any $x\in K$ and $t\ge 0$.

\begin{definition}\label{D:ES-separation}
 Let $K\subset \RR^n$ be an invariant compact subset for $\Phi_t$. The linear skew-flow $(\Phi_t,\,D\Phi_t)$ admits a {\bf $k$-exponential separation along $K$} (for short, $k$-exponential separation), if there are $k$-dimensional continuous bundle $K\times (E_x)$ and $(n-k)$-dimensional continuous bundle $K\times (F_x)$ such that

{\rm (i)} $\RR^n=E_x\oplus F_x$, for any $x\in K$;

{\rm (ii)} $D_x\Phi_tE_x=E_{\Phi_t(x)}$, $D_x\Phi_tF_{x}\subset F_{\Phi_t(x)}$ for any $x\in K$ and $t>0$;

{\rm (iii)}  there are constants $M>0$ and $0<\gamma<1$ such that $$\norm{D_x\Phi_tw}\leq M\gamma^{t}\norm{D_x\Phi_tv}$$ for all $x\in K$, $w\in F_x\cap S$, $v\in E_{x}\cap S$ and $t\ge 0$, where $S=\{v\in \RR^n:\norm{v}=1\}$.

Let $C\subset \RR^n$ be a solid $k$-cone. If, in addition,

{\rm (iv)} $E_x\subset {\rm Int}C\cup \{0\}$ and $F_x\cap C=\{0\}$ for any $x\in K$,

\noindent then $(\Phi_t,\,D\Phi_t)$ is said to admit a {\bf $k$-exponential separation along $K$ associated with $C$}.
\end{definition}

The following proposition is essentially due to Tere\v{s}\v{c}\'ak \cite{Te1}.

\begin{prop}\label{P:Cones-imply-ES}
Assume {\rm {\bf (FWW)}}. Then $(\Phi_t,D\Phi_t)$ admits a $k$-exponential separation along $K$ associated with $C$ along any invariant compact set $K$.
\end{prop}
\begin{proof}
See Tere\v{s}\v{c}\'ak \cite[Corollary 2.2]{Te1}. One may also refer to Tere\v{s}\v{c}\,{a}k \cite[Theorem 4.1]{Te1} and the same argument in \cite[Proposition A.1]{Mi11}.
\end{proof}

Hereafter, by virtue of Proposition \ref{P:Cones-imply-ES}, we will always assume that $(\Phi_t,\,D\Phi_t)$ admits a $k$-exponential continuous separation $\mathbb{R}^n=E_y\oplus F_y$ on a closed bounded set $K$ with respect to $C$, which satisfies (i)-(iv) in Definition \ref{D:ES-separation}.

\vskip 2mm
We now present several critical properties of such $k$-exponential continuous separation with respect to $C$. Given any point $v\in \RR^n$ and any subset $B\subset \RR^n$, we let $d(v,B)=\inf_{w\in B}\norm{v-w}$.

\begin{lemma}\label{vector-in-C-farawayfrom-V1-delta} There exists a constant $\delta^{\prime}>0$ such that $$\{v\in \RR^n: d(v,\,E_y\cap S)\leq \delta^{\prime}\}\subset \text{Int}\,C \,\,\, \text{ for any }y\in K.$$
\end{lemma}
\begin{proof}
Let $M=\bigcup_{y\in K}(E_y\cap S)$. Clearly, $M\subset {\rm Int}C$. We will prove that $M$ is compact. In fact, given any sequence $\{w_n\}\subset M$, there is a sequence $\{y_n\}\subset K$ such that $w_n\in E_{y_n}\cap S$ for every $n\ge 1$. Without loss of generality, we assume that $y_n\to y\in K$. Since the map $y\mapsto E_y\cap S$ is continuous, one can find a sequence  $\{v_n\}\subset E_y\cap S$ such that $\norm{w_n-v_n}\rightarrow 0$ as $n\to \infty$.
Then the compactness of $E_y\cap S$ implies that there is a subsequence $v_{n_k}$ of $v_n$ such that
$v_{n_k}\to w\in E_y\cap S$. So, we have $w_{n_k}\to w\in M$; and hence, $M$ is compact. Therefore, one can find a $\delta'>0$ such that $\{v\in \RR^n: d(v,\,E_y\cap S)\leq \delta^{\prime}\}\subset \text{Int}C$ for any $y\in K$.
\end{proof}

\begin{lemma}\label{ES-distance} There exists a constant $\delta^{\prime\prime}>0$ such that $d(v,C)>\delta^{\prime\prime}$ for any $v\in \bigcup_{y\in K} (F_y\cap S)$.
\end{lemma}

\begin{proof} Let $N=\bigcup_{y\in K} (F_y\cap S)$. By repeating the same argument in the proof of Lemma \ref{vector-in-C-farawayfrom-V1-delta}, one can show that $N$ is compact.
Recall that $C$ is closed and $(F_y\cap S)\cap C=\emptyset$, for any $y\in K$. Then, for any $v\in N$, there exists an open ball $B_{\delta(v)}(v)$ with radius $\delta(v)>0$ such that $B_{\delta(v)}(v)\cap C=\emptyset$. Due to the compactness of $N$, we can obtain a $\delta^{\prime\prime}>0$ such that $d(v,C)>\delta^{\prime\prime}$ for any $v\in \cup_{y\in\omega(x)} (F_y\cap S)$, completing the proof.
\end{proof}

For each $y\in K$, we further denote by $P_{y}: X \longmapsto E_{y}$ the natural projection along $F_{y}$, and $Q_{y}=I-P_{y}$ as the projection onto $F_{y}$. We have the following properties for the projections:

\begin{lemma}\label{L:ES-Vs-Cone}
{\rm (i)} The projections $P_y$ and $Q_y$ are bounded uniformly for $y\in K$.

{\rm (ii)}  There exists a constant $C_1>0$ such that, if $v\in \mathbb{R}^n\setminus\{0\}$ satisfies $\norm{P_y(v)}\geq C_1 \norm{Q_y(v)}$ for some $y\in K$, then $v\in \text{Int}\,C$.

{\rm (iii)} For any $v\in C\setminus\{0\}$, there exists a constant $\delta_3>0$ such that $\norm{Q_y(v)}\leq \delta_3\norm{P_y(v)}$ for any $y\in K$.
\end{lemma}
\begin{proof}
(i). For any $y\in K$ and $v\in \RR^n$, we write $v=P_y(v)+Q_y(v)$, where $P_y(v)\in E_y$ and $Q_y(v)\in F_y$. So, for any $v\in \RR^n$ with $P_y(v)\ne 0$, one has
$$\frac{\norm{v}}{\norm{P_y(v)}}=\norm{\frac{Q_y(v)}{\norm{P_y(v)}}
+\frac{P_y(v)}{\norm{P_y(v)}}}\ge d(\frac{Q_y(v)}{\norm{P_y(v)}}, E_y\cap S).$$
Note that $\frac{Q_y(v)}{\norm{P_y(v)}}\in F_y$. Then it follows from Lemma $\ref{vector-in-C-farawayfrom-V1-delta}$ and $F_y\cap C=\{0\}$ that $\frac{\norm{v}}{\norm{P_y(v)}}\ge \delta^{\prime}$. This implies that $ \norm{P_y}\le \frac{1}{\delta^{\prime}}$ for any $y\in K$. Hence, $\norm{Q_y}\le 1+\frac{1}{\delta^{\prime}}$ for any $y\in K$.

\vskip 2mm

(ii). Let $C_1=\dfrac{2}{\delta^{\prime}}$, where $\delta^\prime>0$ is defined as in (i) and Lemma $\ref{vector-in-C-farawayfrom-V1-delta}$. For any $v\in \mathbb{R}^n\setminus\{0\}$ satisfying $\norm{P_y(v)}\geq C_1 \norm{Q_y(v)}$ for some $y\in K$, it is clear that $P_y(v)\ne 0$ (otherwise, $v=0$, a contradiction). So, we define the nonzero vector
$w=v/\norm{P_{y}(v)}$, and obtain $$\begin{aligned}\norm{\frac{w}{\norm{w}}-\frac{P_{y}(v)}{\norm{P_{y}(v)}}}
&\leq\norm{\frac{w}{\norm{w}}-w}+\norm{\frac{Q_{y}(v)}{\norm{P_{y}(v)}}}\\
 &\leq\norm{w}\cdot\frac{\abs{1-\norm{w}}}{\norm{w}}+\frac{\norm{Q_{y}(v)}}{\norm{P_{y}(v)}}\\
 &\leq\abs{1-\norm{w}}+\frac{\norm{Q_{y}(v)}}{\norm{P_{y}(v)}}\leq 2\frac{\norm{Q_{y}(v)}}{\norm{P_{y}(v)}}.\end{aligned}$$
Noticing that $d(\frac{w}{\norm{w}},E_y\cap S)\leq\norm{\frac{w}{\norm{w}}-\frac{P_{y}(v)}{\norm{P_{y}(v)}}}$, we have $$d(\frac{w}{\norm{w}},E_y\cap S^1)\leq 2\frac{\norm{Q_{y}(v)}}{\norm{P_{y}(v)}}=\frac{2}{C_1}=\delta^\prime.$$
It then follows from Lemma $\ref{vector-in-C-farawayfrom-V1-delta}$ that $w\in \text{Int}\,C$. Therefore, $v\in\text{Int}\,C$.

(iii). Let $\delta_3=\dfrac{1}{\delta^{\prime\prime}}$, where $\delta^{\prime\prime}>0$ is defined as in Lemma \ref{ES-distance}.
Given any $v\in C\setminus\{0\}$ and any $y\in K$, we write $v=P_y(v)+Q_y(v)$. So, by
 Lemma $\ref{ES-distance}$, we have
 $$\norm{P_y(v)}=\norm{v-Q_y(v)}\ge d(Q_y(v),C)\ge \norm{Q_y(v)}\delta^{\prime\prime},$$ which entails that $\norm{Q_y(v)}\leq \delta_3\norm{P_y(v)}$. This completes the proof.
\end{proof}

\vskip 3mm

In the second part of this section, we will introduce {\it the $k$-Lyapunov exponent} for
$(\Phi_t,\,D\Phi_t)$ on $K$ which admits the $k$-exponential separation $K\times \RR^n=E\oplus F$, where $E=K \times (E_x)$ and $F=K \times (F_x)$.

For each $x\in K$, define {\it the $k$-Lyapunov exponent} as
\begin{equation}\label{D:k-Lyapu}
\lambda_{kx}=\limsup_{t\to +\infty}\dfrac{\log m(D_x\Phi_t|_{_{E_x}})}{t},
\end{equation}
where
$m(D_x\Phi_t|_{_{E_x}})=\inf\limits_{v\in E_x\cap S}\norm{D_x\Phi_tv}$ is
the infimum norm of $D_x\Phi_t$ restricted to $E_x$.
A point $x\in K$ is called {\it a regular point} if
$\lambda_{kx}=\lim\limits_{t\to +\infty}\dfrac{\log m(D_x\Phi_t|_{_{E_x}})}{t}$.
\vskip 3mm

\begin{lemma}\label{E-S-and-Lya-exponent} Let $x\in K$. Then

{\rm (i)} If $w\in F_x\setminus\{0\}$, then $\lambda(x,w)\leq \lambda_{kx}+\log(\gamma)$, where $\lambda(x,w)=\limsup_{t\to \infty}t^{-1}\log\norm{D_z\Phi_t v}.$

{\rm (ii)} Let $x$ be a regular point. If $\lambda_{kx}\leq 0$, then there exists a number $\beta\in (\gamma,1)$ satisfying the following property: For any $\epsilon>0$,  there is a constant $C_{\epsilon}>0$ such that $$\norm{D_{\Phi_{t_1} (x)}\Phi_{t_2}w}\leq C_{\epsilon}e^{\epsilon t_1}\beta^{t_2}\norm{w}$$ for any $w\in F_{\Phi_{t_1}(x)}\setminus\{0\}$ and $t_1,t_2>0$.
\end{lemma}

\begin{proof} (i).  Since $(\Phi_t, D\Phi_t)$ admits a $k$-dimensional exponential separation $K\times \RR^n=E\oplus F$, we have $$\begin{aligned} \norm{D_x\Phi_tw}\leq M\gamma^t\frac{\norm{w}}{\norm{v}}\norm{D_x\Phi_tv}   \end{aligned}$$ for any $w\in F_x\setminus\{0\}$, $v\in E_x\setminus\{0\}$ and $t>0$, which implies $\norm{D_x\Phi_tw}\leq M\gamma^t\norm{w}\cdot m(D_x\Phi_t\mid_{_{E_x}})$. Therefore, $$ \begin{aligned} \lambda(x,w)&=\limsup\limits_{t\to \infty}\frac{\log\norm{D_x\Phi_tw}}{t}\leq\limsup\limits_{t\to\infty}
\frac{\log(M\gamma^t\norm{w}\cdot m(D_x\Phi_t\mid_{E_x}))}{t}\\ &= \limsup\limits_{t\rightarrow+\infty}\frac{\log m(D_x\Phi_t\mid_{E_x})}{t}+\log(\gamma)\leq \lambda_{kx}+\log(\gamma).\end{aligned}$$

(ii). Take a number $\beta\in (\gamma,1)$ and fix any $\epsilon\in (0,\log(\frac{\beta}{\gamma})]$. Since the point $x$ is regular,  for any $\eps>0$, there exists some $T^{\epsilon}_1>0$ such that $$ \lambda_{kx}-\frac{\eps}{2}\leq \frac{\log m(D_x\Phi_t\mid_{E_x})}{t} \leq\lambda_{kx}+\frac{\eps}{2}$$ for any $t\ge T^{\epsilon}_1$; and hence, $$\frac{m(D_x\Phi_{t_1+t_2}\mid_{E_x})}{m(D_x\Phi_{t_1}\mid_{E_x})}\leq e^{\lambda_{kx}t_2}e^{\epsilon t_1}e^{\frac{\epsilon}{2} t_2}$$ for any $t_1\ge T^{\epsilon}_1$ and $t_2\ge 0$. Note also that
$$m(D_x\Phi_{t_1+t_2}\mid_{E_{x}})\ge m(D_{\Phi_{t_1}(x)}\Phi_{t_2}\mid_{E_{\Phi_{t_1}(x)}})\cdot m(D_x\Phi_{t_1}\mid_{E_{x}}),\,\,\,\,\text{ for } t_1,t_2\ge 0.$$
Then, together with the $k$-exponentially separated property, we have
$$\begin{aligned} \norm{D_{\Phi_{t_1}(x)}\Phi_{t_2}w}&\leq M\gamma^{t_2}\norm{w}\cdot m(D_{\Phi_{t_1}(x)}\Phi_{t_2}\mid_{E_{\Phi_{t_1}(x)}})\\
&\leq M\gamma^{t_2}\norm{w}\cdot\frac{m(D_x\Phi_{t_1+t_2}\mid_{E_{x}})}{m(D_x\Phi_{t_1}\mid_{E_{x}})}
\leq M\gamma^{t_2}\norm{w}e^{\lambda_{kx}t_2}e^{\epsilon t_1}e^{\frac{\epsilon}{2} t_2}\end{aligned}$$ for any $t_1\ge T^{\epsilon}_1$, $t_2\ge 0$ and $w\in F_{\Phi_{t_1}(x)}\setminus\{0\}$. Recall that $\lambda_{kx}\le 0$. Then
$\norm{D_{\Phi_{t_1} (x)}\Phi_{t_2}w}\le Me^{\epsilon t_1}(e^{\frac{\epsilon}{2}}\gamma)^{t_2}\norm{w}.$ Then, we obtain
\begin{equation}\label{E:Lypa-growth-control}
\norm{D_{\Phi_{t_1} (x)}\Phi_{t_2}w}\leq Me^{\epsilon t_1}\beta^{t_2}\norm{w},\,\,\,\,\text{ for any } t_1\ge T^{\epsilon}_1,\, t_2\ge 0\,\text{ and } w\in F_{\Phi_{t_1}(x)}\setminus\{0\}.
\end{equation}
Now, for $t_1\in [0,T^{\epsilon}_1]$, we define $\chi_{\epsilon}^1=\max\limits_{\substack{0\leq t_1\leq T^{\epsilon}_1\\0\leq t_2\leq T^{\epsilon}_1-t_1}}\{\norm{D_{\Phi_{t_1}(x)}\Phi_{t_2}}\}$. By the smoothness of $\Phi_t$,  $\chi_{\epsilon}^1<+\infty$. Let also $\chi_{\epsilon}^2=\chi^1_{\epsilon}\cdot\max\limits_{\substack{0\leq t_1\leq T^{\epsilon}_1\\0\leq t_2\leq T^{\epsilon}_1-t_1}}\{e^{-\epsilon t_1}\beta^{-t_2}\}$ and $\chi_{\epsilon}^3=M\chi_{\epsilon}^1(e^{\epsilon}\gamma^{-1})^{T^{\epsilon}_1}$. Then, the following properties hold:

$(P_1)$ $\norm{D_{\Phi_{t_1}(x)}\Phi_{t_2}w}\leq \chi^2_{\epsilon}e^{\epsilon t_1}\beta^{t_2}\norm{w}$ for any $0\leq t_1\leq T^{\epsilon}_1$, $0\leq t_2\leq T^{\epsilon}_1-t_1$ and $w\in F_{\Phi_{t_1}(x)}\setminus\{0\}$;

$(P_2)$ For any $0\leq t_1\leq T^{\epsilon}_1$, $t_2>T^{\epsilon}_1-t_1$ and $w\in F_{\Phi_{t_1}(x)}\setminus\{0\}$, one has $$\begin{aligned}\norm{D_{\Phi_{t_1}(x)}\Phi_{t_2}w}&=
\norm{D_{\Phi_{T^{\epsilon}_1}(x)}\Phi_{t_1+t_2-T^{\epsilon}_1}D_{\Phi_{t_1}(x)}\Phi_{T^{\epsilon}_1-t_1}w}\leq \chi^3_{\epsilon}e^{\epsilon t_1}\beta^{t_2}\norm{w}.\end{aligned}$$

\noindent Therefore, combining with \eqref{E:Lypa-growth-control}, we obtain that $\norm{D_{\Phi_{t_1} (x)}\Phi_{t_2}w}\leq \max\{M,\chi^2_{\epsilon},\chi^{3}_{\epsilon}\}\cdot e^{\epsilon t_1}\beta^{t_2}\norm{w}$ for any $w\in F_{\Phi_{t_1}(x)}\setminus\{0\}$ and $t_1,t_2>0$.  The proof is complete with  \begin{equation}
C_{\epsilon} = \begin{cases}
\max\{M,\chi^2_{\epsilon},\chi^{3}_{\epsilon}\}, & \epsilon\in(0,\log(\frac{\beta}{\gamma})) \\
\max\{M,\chi^2_{\log(\frac{\beta}{\gamma})},\chi^{3}_{\log(\frac{\beta}{\gamma})}\}, & \epsilon\in[\log(\frac{\beta}{\gamma}),\infty).
\end{cases}
\end{equation}
\end{proof}

\begin{rem}
{\rm When $k=1$ in \eqref{D:k-Lyapu}, $\lambda_{kx}$ naturally reduces to $\lambda_{1x}=\limsup\limits_{t\to +\infty}\dfrac{\log \norm{D_x\Phi_tv_x}}{t}$, where $\norm{v_x}=1$ with $E_x={\rm span}\{v_x\}.$ In general, $\lambda_{1x}$ is referred as {\it the first (or principal) Lyapunov exponent} (see, for example, {\rm \cite{PT1}}).
}
\end{rem}

\section{$k$-Stability and Pseudo-ordered Orbits}

We will investigate in this section the generic behavior of the orbits for the smooth flow $\Phi_t$ under the fundamental assumption {\bf (FWW)}.

For this purpose, throughout this section we fix an open set $D\subset \RR^n$ that is $\omega$-compact, and focus on the types of the orbits starting from $D$. For any $x\in D$, $(\Phi_t,\,D\Phi_t)$ admits a $k$-exponential continuous separation  $\omega(x)\times \RR^n=E\oplus F$ with respect to $C$. Motivated by Pol\'{a}\v{c}ik and Tere\v{s}\v{c}\'ak \cite{PT1}, we will classify the dynamics for the orbit of $x$ according to
 the $k$-Lyapunov exponents on $\omega(x)$.
To be more specific, we firstly show that if $\omega(x)$ contains a regular point $z$ such that $\lambda_{kz}\le 0$, then either $x\in Q$ or $\omega(x)$ is a singleton (see Theorem \ref{omega-lambdak<=0}). Secondly, we prove that if $\lambda_{kz}>0$ for any $z\in \omega(x)$, then $x$ is highly unstable (see Lemma \ref{Distance-nonlinear-system}), and meanwhile, it belongs to the closure $\overline{Q}$ (see Theorem \ref{omega-lambdak>0}). Finally, we show that if $\lambda_{k\tilde{z}}>0$ for any regular points $\tilde{z}\in \omega(x)$ but there is an irregular point $z$ with $\lambda_{kz}\leq 0$, then $x\in \overline{Q}$ (see Theorem \ref{regular-lambdak>0}).

\vskip 2mm
We begin with the following lemma, which extends the nonlinear dynamics nearby an equilibrium, the  linearization of which is governed by the Perron-Frobenius Theorem, to that nearby a regular point.

\begin{lemma}\label{oribit-lambdak<0} Let $x$ be a regular point. If $\lambda_{kx}\leq 0$, then there exists an open neighborhood $\mathcal{V}$ of $x$ such that for any $y\in \mathcal{V}$, one of two following properties holds:
\vskip 1mm
$(a)$ $\norm{\Phi_{t}(x)-\Phi_{t}(y)}\rightarrow 0$ as $t\rightarrow +\infty$;

\vskip 2mm
$(b)$ There exists $T>0$ such that $\Phi_{T}(x)-\Phi_{T}(y)\in  C$; an hence, $\Phi_{t}(x)-\Phi_{t}(y)\in \text{Int}\, C$ for any $t>T$.
\end{lemma}

\begin{proof} In order to prove this lemma, we first assert that there exists an open neighborhood $\mathcal{V'}$ of $x$ such that for any $y\in \mathcal{V'}$, one of two following properties must hold:
\vskip 1mm
 $(a')$ $\norm{\Phi_{n}(x)-\Phi_{n}(y)}\rightarrow 0$ as $n\rightarrow +\infty$;
  \vskip 2mm
 $(b')$ There exists $N\in \mathbb{N}$ such that $\Phi_{N}(x)-\Phi_{N}(y)\in \text{Int}\,C$.

 Before proving this assertion, we will show how it implies this lemma. In fact, fix any $y\in \mathcal{V}$ and suppose $(b)$ does not hold. Then $(b^\prime)$ does not hold for $y$. By virtue of the assertion, one has $\lim\limits_{n\rightarrow +\infty} \norm{\Phi_n(x)-\Phi_n(y)}=0$ as $n\rightarrow +\infty$. For any $t\geq 0$, we write $t=n+\tau$ with $n\in \mathbb{N}$ and $\tau\in[0,1)$. Then $$\norm{\Phi_t(x)-\Phi_t(y)}=\norm{\Phi_{\tau}\circ\Phi_{n}(x)-\Phi_{\tau}\circ\Phi_{n}(y)}
 \leq\sup\limits_{s\in[0,1]}\norm{D\Phi_{s}(\mathcal{B})}\cdot\norm{\Phi_{n}(x)-\Phi_{n}(y)}.$$
 Here, $\mathcal{B}=\{z\in \mathbb{R}^n:\,\norm{z}\leq 1\}$. Therefore, one has $\norm{\Phi_t(x)-\Phi_t(y)}\rightarrow 0$ as $n\rightarrow \infty$, that is, $(a)$ holds. Then, we complete the proof of this lemma.

It remains to prove the assertion. Recall that the linear skew product flow $(\Phi,\,D\Phi)$ on $\overline{O(x)}\times \mathbb{R}^n$ admits $k$-exponential continuous separation as $\overline{O(x)}\times\mathbb{R}^n=\overline{O(x)}\times(E_y)\oplus \overline{O(x)}\times(F_y)$ with $E_y\subset \text{Int}\,C\cup \{0\}$ and $F_y\cap C=\{0\}$ for any $y\in\overline{O(x)}$. Let $P_y$ be the continuous projection onto $E_y$ along $F_y$ and $Q_y=I-P_y$.

\vskip 2mm

For brevity, we write $x_n=\Phi_{n}(x)$, $y_n=\Phi_{n}(y)$ and $g_n(z)=\Phi_{1}(x_n+z)-\Phi_{1}(x_n)-D_{x_n}\Phi_{1}z$. Recall that $\Phi_{t}$ is $C^{1,\alpha}$ and $\overline{O(x)}$ is bounded. Then $$g_{n}(z)=O(\norm{z}^{1+\alpha})$$ as $z\rightarrow 0$ uniformly for $n\in \mathbb{N}$.  Let $z_{n}=y_n-x_n$, then $$z_{n+1}=D_{x_n}\Phi_{1}z_{n}+g_{n}(z_n).$$

\vskip 2mm
By virtue of Lemma \ref{L:ES-Vs-Cone} (ii), the assertion follows from the following statement: There exists an open neighborhood $\mathcal{W}$ of $0$ such that if $z_0\in \mathcal{W}$, then either

$(a^{\prime\prime})$ $z_n\rightarrow 0$ as $n\rightarrow +\infty$; or else

$(b^{\prime\prime})$ There is an $n$ such that $\norm{P_{x_n}z_n}\geq C_1 \norm{Q_{x_n}z_n}$.\\ Here, $C_1$ is defined in Lemma $\ref{L:ES-Vs-Cone}$ (ii).

To prove this statement and the existence of $\mathcal{W}$, we assume Case $(b^{\prime\prime})$ does not hold for any $n\geq 0$, which means $C_1 \norm{Q_{x_n}z_n}\geq \norm{P_{x_n}z_n}$ for any $n\geq 0$. By utilizing $\lambda_{kx}\leq 0$, we will show that $Q_{x_n}z_n\rightarrow 0$ if $z\in \mathcal{W}\triangleq\{z\in \mathbb{R}^n: \,\,\norm{Q_0z_0}<\xi\}$ with $\xi>0$ sufficiently small to be defined later. This then implies $z_n\rightarrow 0$ as $n\rightarrow +\infty$.

For brevity, we hereafter rewrite $D_{x_n}\Phi_{1}$, $P_{x_n}$ and $Q_{x_n}$ as $D_n$, $P_n$ and $Q_n$, respectively. By the invariance of exponential separation, we have $D_{n}P_n=P_{n+1}D_{n}$ and $D_{n}Q_n=Q_{n+1}D_{n}$. So $$P_{n+1}z_{n+1}=D_nP_nz_n+P_{n+1}g_n(z_n)$$ and $$Q_{n+1}z_{n+1}=D_nQ_nz_n+Q_{n+1}g_n(z_n)$$ for any $n\geq 0$. It then follows that $$Q_{n}z_n=T(n,0)Q_0z_0+\Sigma_{k=0}^{n-1}T(n,k+1)Q_{k+1}g_{k}(z_k)$$ for $n\in \mathbb{N}^+$, where $T(n,m)=D_{n-1}\circ D_{n-2}\circ\cdots\circ D_m$ for $n>m$ and $T(m,m)=id_{X}$.

Recall that $g_{n}(z)=\text{O}(\norm{z}^{1+\alpha})$ as $z\rightarrow 0$ uniformly for $n=0,1,2,\,\cdots$. Then there exist constants $r>0$ and $C_2>0$ such that $\norm{g_n(z)}\leq C_2\norm{z}^{1+\alpha}$ for any $n\geq 0$ if $\norm{z}<r$. Together with $C_1\norm{Q_nz_n}\geq \norm{P_nz_n}$, one has that $\norm{z_n}\leq r$ if $\norm{Q_nz_n}\leq\frac{r}{1+C_1}$. Hence, one has that \begin{eqnarray}\label{Tailterm}\norm{g_{n}(z_n)}\leq C_2\norm{z_n}^{1+\alpha}\leq C_2(1+C_1)^{1+\alpha}\norm{Q_nz_n}^{1+\alpha}\end{eqnarray} whenever $\norm{Q_nz_n}\leq\frac{r}{1+C_1}$ for $n\geq 0$.

Recall that $x$ is a regular point with $\lambda_{kx}\leq 0$. Then Lemma \ref{E-S-and-Lya-exponent}(ii) yields that there exists a constant $\beta\in (\gamma,1)$ such that for any $\epsilon>0$, we can find a number $C_{\epsilon}>0$ satisfying $$\norm{T(n,k)w}\leq C_{\epsilon}e^{\epsilon k}\beta^{n-k}\norm{w}$$ for any $w\in F_{x_k}$ with $n\geq k\geq 0$.

Without loss of generality, we assume $C_{\epsilon}\geq1$ and take $\epsilon>0$, $\eta\in (\beta,1)$ such that $e^{\epsilon}\eta^{\alpha}<1$. By Lemma \ref{L:ES-Vs-Cone}(i), one can find an upper bound $C_3>0$ such that $\norm{Q_n}\leq C_3$ for any $n\geq 0$. Let $C_4=C_{\epsilon}(1+C_1)^{1+\alpha}C_2C_3e^{\epsilon}\eta^{-1}$ and define the following neighborhood of $0$ $$\mathcal{W}=\{z\in \mathbb{R}^n:\,\norm{Q_0z_0}<\xi\},$$ where $\xi=\dfrac{\rho}{2C_{\epsilon}}$ with $0<\rho\leq\min\{\frac{r}{1+C_1},\,(\frac{\eta-\beta}{2C_4\eta})^{-\alpha}\}$.\\

We now claim that $\norm{Q_{n}z_n}\leq \rho\eta^{n}$ for any $n\geq0$, where $z_0\in \mathcal{W}$. Clearly, $\norm{Q_0z_0}\leq \frac{\rho}{2C_{\epsilon}}< \rho$. We will prove the claim by induction. Suppose that $\norm{Q_kz_k}\leq\rho\eta^k\leq\rho$ for any $k=0,1,\cdots,n-1$, it suffices to prove $\norm{Q_nz_n}\leq \rho\eta^n$. Noticing that $\rho\leq\frac{r}{1+C_1}$. The inequality (\ref{Tailterm}) yields that $\norm{g_k(z_k)}\leq C_2(1+C_1)^{1+\alpha}\norm{Q_kz_{k}}^{1+\alpha}$ for any $0\leq k\leq n-1$. Consequently,
$$\begin{aligned}\norm{Q_nz_{n}}&\leq \norm{T(n,0)Q_0z_0}+\Sigma_{k=0}^{n-1}\norm{T(n,k+1)Q_{k+1}g_{k}(z_k)}\\
&\leq C_{\epsilon}\beta^n\norm{Q_0z_0}+\Sigma_{k=0}^{n-1}C_{\epsilon}e^{\epsilon(k+1)}\beta^{n-k-1}\norm{Q_{k+1}}\norm{g_{k}(z_k)}\\
&\leq\frac{\beta^n}{2}\rho+C_{\epsilon}(\sup \limits_{0\leq k\leq n-1}\norm{Q_{k+1}})\Sigma_{k=0}^{n-1}e^{\epsilon(k+1)}\beta^{n-k-1}C_2(1+C_1)^{1+\alpha}\norm{Q_kz_k}^{1+\alpha}\\
&\leq\rho\cdot \frac{\beta^n}{2}+\rho\cdot C_4\Sigma_{k=0}^{n-1}(e^{\epsilon}\eta^{\alpha})^k\beta^{n-k-1}\eta^{k+1}\rho^{\alpha}.\end{aligned}$$

Recall that $0<e^{\epsilon}\eta^{\alpha}<1$. One has $$\norm{Q_nz_n}\leq\eta^n[\frac{\rho}{2}\cdot (\frac{\beta}{\eta})^n+\rho C_4\cdot\rho^{\alpha}\frac{\eta}{\eta-\beta}].$$

Together with $\beta<\eta<1$ and $\rho\leq(\frac{\eta-\beta}{2C_4\eta})^{-\alpha}$, we obtain that $\norm{Q_nz_n}\leq\rho\cdot\eta^{n}.$ Thus, we have proved the claim and hence $z_n\rightarrow 0$ as $n\rightarrow \infty$.
\end{proof}

\begin{theorem}\label{omega-lambdak<=0}
 Assume that there exists a regular point $z\in\omega(x)$ satisfying $\lambda_{kz}\leq 0$. Then either $x\in Q$, or $\omega(x)=\{z\}$ is a singleton.
\end{theorem}

\begin{proof} Without loss of generality, we assume that $x$ is not an equilibrium. By Lemma $\ref{oribit-lambdak<0}$, there exists an open neighborhood $\mathcal{V}$ of $z$ such that for any $y\in \mathcal{V}$, one of following properties holds:

$(a)$ $\norm{\Phi_{t}(z)-\Phi_{t}(y)}\rightarrow 0$ as $t\rightarrow +\infty$.

$(b)$ There is $T>0$ such that $\Phi_{T}(z)-\Phi_{T}(y)\in C$.

Now, we define the set $\Gamma=\{t\geq 0:\,\Phi_t(x)\in \mathcal{V}\}$. Since $z\in \omega(x)$, it is clear that $\Gamma\neq \emptyset$.

If $\Phi_{t_0}(x)$ satisfies property $(b)$ for some $t_0\in \Gamma$, then there is some $T>0$ such that $\Phi_{T}(z)-\Phi_{T+t_0}(x)\in C$. Furthermore, by the strong monotonicity of $\Phi_t$, one can find a $T_0>0$ and an open neighborhood $U$ of $z$ and $V$ of $\Phi_{t_0}(x)$, respectively such that $\Phi_{T+T_0}U\thickapprox \Phi_{T+T_0+t_0}V$. Recall that $z\in\omega(x)$. Then, take $t_{1}>t_{0}$ such that $\Phi_{t_{1}}(x)\in U\setminus \{\Phi_{t_0}(x)\}$. Consequently, we have $\Phi_{t_{1}+T+T_0}(x)\thickapprox\Phi_{t_0+T+T_0}(x)$, which implies $O(x)$ is pseudo-ordered.

If $\Phi_{\tau}(x)$ satisfies property $(a)$ for any $\tau\in\Gamma$, then for any $\tau_1,\tau_2\in \Gamma$ with $\tau_1<\tau_2$ satisfying $\Phi_{\tau_1}(x)\neq \Phi_{\tau_2}(x)$, one has $\norm{\Phi_{\tau_1+t}(x)-\Phi_{t}(z)}\rightarrow 0$
and $\norm{\Phi_{\tau_2+t}(x)-\Phi_{t}(z)}\rightarrow 0$ as $t\rightarrow +\infty$. So $\norm{\Phi_{\tau_1+t}(x)-\Phi_{\tau_2+t}(x)}\rightarrow 0$ as $t\rightarrow +\infty$. Fix any $u\in \omega(x)$, there is a sequence $\{s_k\}_{k=1}^{\infty}$ such that $s_k\rightarrow +\infty$ and $\Phi_{s_k}(x)\rightarrow u$ as $k\rightarrow +\infty$. Hence, $\norm{\Phi_{\tau_1+s_k}(x)-\Phi_{\tau_2+s_k}(x)}\rightarrow 0$ as $k\rightarrow +\infty$, which implies that $u=\Phi_{\tau_2-\tau_1}(u)$. By the arbitriness of $u$, it follows that $\omega(x)$ consists of $(\tau_2-\tau_1)$-periodic points.

Now, since $\mathcal{V}$ is open and $x$ is not an equilibrium, one can find $\tau_0\in \Gamma$ and $\epsilon>0$ satisfying: $(i)$. $[\tau_0,\tau_0+\epsilon]\subset \mathbb{T}$; $(ii)$. $\Phi_{s_1}(x)\neq \Phi_{s_2}(x)$ for any $s_1,s_2\in [\tau_0,\tau_0+\epsilon]$ with $s_1\neq s_2$.

By repeating the argument in the previous paragraph, one can obtain that $\omega(x)$ consists of $s$-periodic point for any $s\in [0,\epsilon]$. Thus, $\omega(x)$ consists of equilibria. In particular, $z$ is an equilibrium. Recall that $\Phi_{\tau}(x)$ satisfies property $(a)$ for $\tau\in\Gamma$. Then one has $\norm{\Phi_{\tau+t}(x)-z}=\norm{\Phi_{\tau+t}(x)-\Phi_t(z)}\rightarrow 0$ as $t\rightarrow +\infty$. Therefore, one has $\omega(x)=\{z\}$, which completes the proof.
\end{proof}
\vskip 2mm
Now we will discuss the case that $\lambda_{kz}>0$ for all point $z\in \omega(x)$. Before going further,  we here present the following two technical lemmas.

\begin{lemma}\label{expand-in-Ez} Let $\delta_3>0$ be defined in Lemma {\rm\ref{L:ES-Vs-Cone}(iii)}. If $\lambda_{kz}>0$ for any $z\in \omega(x)$, then there is a locally constant (hence bounded) function $\nu(z)$ on $\omega(x)$ such that
\begin{description}
\item[{\rm (i).}] $\norm{D_z\Phi_{\nu(z)}w_F}< \frac{1}{2\delta_3}\norm{D_z\Phi_{\nu(z)}w_E}$,
\item[{\rm (ii).}] $\norm{D_z\Phi_{\nu(z)}w_E}> 4(1+\delta_3)$
\end{description}
for any $z\in \omega(x)$ and any unit vector $w_E\in E_z,w_F\in F_z$.
\end{lemma}
\begin{proof} Since $(\Phi_t, D\Phi_t)$ admits a $k$-dimensional exponential continuous separation along $\omega(x)$ as $\omega(x)\times \RR^n=\omega(x)\times (E_y)\bigoplus \omega(x)\times (F_y)$, it is clear that  there is a number $T>0$ such that $\norm{D_z\Phi_{t}w_F}< \frac{1}{2\delta_3}\norm{D_z\Phi_{t}w_E}$ for $w_E\in E_z\cap S$, $w_F\in F_z\cap S$ and any $t>T$.

Moreover, by the definition of $\lambda_{kz}$,  for each $z\in \omega(x)$, there is a sequence $t_n\to +\infty$ such that $$\norm{D_z \Phi_{t_n}w_E}> e^{\frac{\lambda_{kz}}{2}t_n}$$ for any $w_E\in E_z\cap S$. Since
$\lambda_{kz}>0$, one can find an integer $N_z>0$ such that $\norm{D_z\Phi_{t_n}w_E}> 4(1+\delta_3)$ for any $t_n>N_z$ and $w_E\in E_z\cap S$.

Therefore, for each $z\in \omega(x)$, one can associate with a number $\nu(z)\ge \max\{T,N_z\}>0$ such that
 $\norm{D_z\Phi_{\nu(z)}w_F}< \frac{1}{2\delta_3}\norm{D_z\Phi_{\nu(z)}w_E}$ and $\norm{D_z\Phi_{\nu(z)}w_E}> 4(1+\delta_3)$ for any $w_E\in E_z\cap S$, $w_F\in F_z\cap S$ and $z\in \omega(x)$.  Moreover, together with the compactness of $\omega(x)$, the continuity of $z\mapsto D_z\Phi_\nu$ implies that one can further take such $\nu(z)$ as a locally constant (hence bounded) function. This completes the proof.
\end{proof}

\begin{lemma}\label{Distance-nonlinear-system} Assume that $\lambda_{kz}>0$ for any $z\in \omega(x)$. There exists a constant $\delta>0$ such that $$\limsup_{t\to +\infty}\norm{\Phi_t(y)-\Phi_t(x)}\ge \delta,$$ whenever $y$ satisfies $y\neq x$ and $y\thicksim x$.

\end{lemma}

\begin{proof}
Let $m$ be an upper bound of the function $\nu(\cdot)$ on $\omega(x)$ defined in Lemma $\ref{expand-in-Ez}$. Due to the compactness of $\omega(x)$, one can choose a $\delta_0>0$ so small that $\norm{D_u\Phi_{\nu}-D_v\Phi_{\nu}}\leq \frac{1}{2}$ for any $u,v\in \text{Co}\{u\in \mathbb{R}^n: d(u,\omega(x))\leq 1\}$ with $\norm{u-v}\le \delta_0$ and any $\nu\in [0,m]$, where ``Co" means the convex hull.

Let $\delta=\delta_0/2>0$. We will show that such $\delta$ satisfies this lemma. Suppose that one can find some
$y\neq x$ with $y\thicksim x$ such that $\limsup_{t\to +\infty}\norm{\Phi_t(y)-\Phi_t(x)}<\delta.$ Then
there exists an $N_1>0$ such that $\norm{\Phi_t(y)-\Phi_t(x)}<\delta$ for any $t\ge N_1$.
Since $\Phi_t(x)$ is attracted to $\omega(x)$, one can choose $z_{t}\in \omega(x)$ such that $\norm{\Phi_t (x)-z_t}\to 0$ as $t\to \infty$. Hence, there exists a number $N_2>N_1$ such that
\begin{equation}\label{E:asym-small}
\norm{\Phi_t (y)-\Phi_t (x)}\leq \delta\,\,\text{ and }\,\, \norm{\Phi_t (x)-z_t}\leq \delta, \quad \text{ for any }t\ge N_2.
\end{equation}

Let also $\tau_1=1$ and $\tau_{k+1}=\tau_k+\nu(z_{\tau_k})$ for $k=1,2,\cdots$.
 Denoted by $y_{\tau_k}=\Phi_{\tau_k}(y)$ and $x_{\tau_k}=\Phi_{\tau_k}(x)$. Then we have that $$\begin{aligned}y_{\tau_{k+1}}-x_{\tau_{k+1}}
&=D_{z_{\tau_k}}\Phi_{\nu(z_{\tau_k})}(y_{\tau_k}-x_{\tau_k})\\
&+\int_{0}^{1}\big[D_{x_{\tau_k}+s(y_{\tau_k}-x_{\tau_k})}\Phi_{\nu(z_{\tau_k})}
-D_{z_{\tau_k}}\Phi_{\nu(z_{\tau_k})}\big](y_{\tau_k}-x_{\tau_k})ds.\end{aligned}$$

By \eqref{E:asym-small}, one can find a positive number $N>0$ such that $\norm{y_{\tau_k}-x_{\tau_k}}\leq \delta$ and $\norm{x_{\tau_k}-z_{\tau_k}}\leq \delta$ for any $k\geq N$. Hence, $\norm{x_{\tau_k}+s(y_{\tau_k}-x_{\tau_k})-z_{\tau_k}}\le 2\delta=\delta_0$ for any $s\in [0,1]$ and $k\geq N$. As a consequence, we have
\begin{equation}\label{E:Holder-part}
\norm{\int_{0}^{1}\big[D_{x_{\tau_k}+s(y_{\tau_k}-x_{\tau_k})}
\Phi_{\nu(z_{\tau_k})}-D_{z_{\tau_k}}\Phi_{\nu(z_{\tau_k})}\big](y_{\tau_k}-x_{\tau_k})ds}
<\frac{1}{2}\norm{(y_{\tau_k}-x_{\tau_k})}
\end{equation}
for any $k\geq N$.

On the other hand, since $y_{\tau_k}-x_{\tau_k}\in C\setminus \{0\}$, Lemma \ref{L:ES-Vs-Cone}(iii)
directly entails that
\begin{equation}\label{E:E-Projection-control}
\norm{y_{\tau_k}-x_{\tau_k}}\leq (1+\delta_3)\norm{P_{\tau_k}(y_{\tau_k}-x_{\tau_k})}\,\,\,\, \text{ and }
\,\,\,\, \frac{\norm{Q_{\tau_k}(y_{\tau_k}-x_{\tau_k})}}{\norm{P_{\tau_k}(y_{\tau_k}-x_{\tau_k})}}\le \delta_3
\end{equation}
for any $k\ge 1$. It then follows from Lemma \ref{expand-in-Ez}(i)-(ii) and \eqref{E:E-Projection-control} that $$\begin{aligned}\norm{D_{z_{\tau_k}}\Phi_{\nu(z_{\tau_k})}(y_{\tau_k}-x_{\tau_k})}&\geq \norm{D_{z_{\tau_k}}\Phi_{\nu(z_{\tau_k})}P_{\tau_k}(y_{\tau_k}-x_{\tau_k})}- \norm{D_{z_{\tau_k}}\Phi_{\nu(z_{\tau_k})}Q_{\tau_k}(y_{\tau_k}-x_{\tau_k})}\\
&\ge \norm{D_{z_{\tau_k}}\Phi_{\nu(z_{\tau_k})}P_{\tau_k}(y_{\tau_k}-x_{\tau_k})}\cdot
\left[1-\frac{\norm{D_{z_{\tau_k}}\Phi_{\nu(z_{\tau_k})}Q_{\tau_k}(y_{\tau_k}-x_{\tau_k})}}
{\norm{D_{z_{\tau_k}}\Phi_{\nu(z_{\tau_k})}P_{\tau_k}(y_{\tau_k}-x_{\tau_k})}}\right]\\
&\geq\frac{1}{2}\norm{D_{z_{\tau_k}}\Phi_{\nu(z_{\tau_k})}P_{\tau_k}(y_{\tau_k}-x_{\tau_k})}\geq2(1+\delta_3)\norm{P_{\tau_k}(y_{\tau_k}-x_{\tau_k})}\\
&\geq 2\norm{y_{\tau_k}-x_{\tau_k}}\end{aligned}$$ for any $k\geq N$.
Together with \eqref{E:Holder-part}, this implies that
$$\norm{y_{\tau_{k+1}}-x_{\tau_{k+1}}}\geq \frac{3}{2}\norm{y_{\tau_k}-x_{\tau_k}}$$ for any $k\geq N$. This contradicts $\limsup\limits_{t\to +\infty}\norm{\Phi_t(y)-\Phi_t(x)}<\delta.$ Thus, we have proved the Lemma.
\end{proof}

\begin{theorem}\label{omega-lambdak>0}
If $\lambda_{kz}>0$ for any $z\in \omega(x)$, then we have $x\in \overline{Q}$.
\end{theorem}
\begin{proof}
Choose any sequence $\{x_n\}_{n=1}^\infty\subset \mathbb{R}^n$ approaching $x$ with $x_n\thicksim x$. Since $x\in D$ and $D$ is open and $\omega$-compact,  $\overline{\cup_{n\geq 1}\omega(x_n)}$ is a nonempty compact set. Define $\omega_{\rm c}(x)=\bigcap_{n\geq 1}\overline{\bigcup_{m\geq n}\omega(x_m)}$. Then, it is clear that
$\omega_{\rm c}(x)$ is a nonempty and compact subset which is invariant with respect to $\Phi_t$. Moreover, $q\in\omega_{\rm c}(x)$ if and only if there are two subsequences $n_k\to +\infty$ and $q_{n_k}\in\omega(x_{n_k})$
such that $q_{n_k}\to q$ as $k\to \infty$.

Since $\lambda_{kz}>0$ for any $z\in \omega(x)$, it follows from
Lemma $\ref{Distance-nonlinear-system}$ that for each $n\ge 1$, there exist some $p_n\in \omega(x)$ and $q_n\in\omega(x_{n})$ such that $p_n\thicksim q_n$ and $\norm{p_n-q_n}\geq \delta$. Choose a subsequence $n_k\to \infty$, if necessary, such that $p_{n_k}\to p\in\omega(x)$ and $q_{n_k}\to q\in \omega_{\rm c}(x)$ as $k\to +\infty$; and hence, one can further find a susequence $s_k\to \infty$ such that $\Phi_{s_k}(x_{n_k})\to q$ as $k\to +\infty$.

Clearly, $p\thicksim q$ and $\norm{p-q}\geq \delta$. Since $\Phi_{t}$ is strongly monotone with respect to $k$-cone $C$, we have $\Phi_1(p)\thickapprox\Phi_{1}(q)$. Choose some open neighborhoods $U$ and $V$ of $\Phi_1(p)$ and $\Phi_1(q)$, respectively, such that $U\thickapprox V$ and $U\cap V=\emptyset$. Then $\Phi_{s_k+1}(x_{n_k})\in V$ for all $k$ sufficiently large. Recalling that $x_n\to x$, we can take a number $T>0$ s.t $\Phi_{1+T}(x_{n_k})\in U$ for all $k$ sufficiently large. Consequently,
$\Phi_{1+T}(x_{n_k})\thickapprox\Phi_{1+s_k}(x_{n_k})$ for all large $k$. This entails that the orbit $O(x_{n_k})$ is pseudo-ordered, i.e., $x_{n_k}\in Q$ for all large $k$, which implies that $x\in\overline{Q}$.
\end{proof}

\vskip 3mm

Motivated by Theorems \ref{omega-lambdak<=0} and \ref{omega-lambdak>0}, we define
the set of {\it regular points} on $\omega(x)$ as:
 \begin{equation}\label{E:rg-point}
\omega_0(x)=\{z\in\omega(x):\,z\,\, \text{is a regular point}\}.
\end{equation}
Due to the Multiplicative Ergodic Theorem (cf. \cite[Theorem 2.1]{John-Pa-S}) and the similar argument in \cite[Proposition 4.1]{Vi}, one has $\omega_0(x)$ is non-empty. Moreover, it is easy to see that any equilibrium in $\omega(x)$ is regular, and hence, is contained in $\omega_0(x)$.

\begin{theorem}\label{regular-lambdak>0} Assume that $\lambda_{k\tilde{z}}>0$ for any $\tilde{z}\in \omega_0(x)$. If there exists some $z\in\omega(x)\setminus \omega_0(x)$ such that $\lambda_{kz}\leq 0$, then $x\in \overline{Q}$.
\end{theorem}

\begin{proof} For any $y\in \omega(x)$, we define a vector $v_y:=\frac{d}{dt}|_{_{t=0}}\Phi_t(y)$ in $\RR^n$.  Clearly, the map $y\mapsto v_y$ is continuous. Moreover, we have $D_y \Phi_t(v_y)=v_{\Phi_t(y)}$ for all $t\in \mathbb{R}$; and hence, $v_y=0$ if and only if $y$ is an equilibrium.

Since $z\in\omega(x)\setminus \omega_0(x)$,  we have $v_z\neq 0$.  Then
\begin{equation}\label{E:center-Lypo} \lambda(z,v_z)=\limsup\limits_{t\to+\infty}\frac{\log\norm{D_z\Phi_t(v_z)}}{t}
=\limsup\limits_{t\to+\infty}\frac{\log\norm{v_{\Phi_t(z)}}}{t}.
\end{equation}
Recall that $y\mapsto v_y$ is continuous and $\omega(x)$ is compact. Then $\norm{v_{\Phi_t(z)}}$ is bounded uniformly for any $t\geq 0$. This implies that $\lambda(z,v_{z})\leq 0.$ We will consider the following two cases: (A1). $\lambda(z,v_{z})=0$; (A2). $\lambda(z,v_{z})<0$, respectively.

(A1). If $\lambda(z,v_z)=0$, then by $\lambda_{kz}\leq0$ and Lemma $\ref{E-S-and-Lya-exponent}$ (i), one has $v_z\notin F_z$. So, one can write $v_z$ as $v_z=\alpha v +\beta w$, where $v\in E_z\setminus\{0\}$, $w\in F_z$ and $\alpha\neq 0$. By the exponential separation property and Lemma \ref{L:ES-Vs-Cone}(ii), there is a constant $T>0$ such that $D_z\Phi_t(v_z)\in \text{Int}\, C$ for any $t>T$. This implies $z\in Q$. Together with the strong monotonicity of $\Phi$, one can find two different constants $\tau_1,\tau_2>0$ such that $\Phi_{\tau_1}z\thickapprox\Phi_{\tau_2}z$. Hence, there exist open neighborhoods $U_1$ and $U_2$ of $\Phi_{\tau_1}z$ and $\Phi_{\tau_2}z$, respectively such that $U_1\thickapprox U_2$ and $U_1\cap U_2=\emptyset$. Since $z\in \omega(x)$, we obtain that $x\in Q$.

(A2). If $\lambda(z,v_z)<0$, then \eqref{E:center-Lypo} implies that there exists some $M>0$ such that $\norm{v_{\Phi_t(z))}}\leq Me^{\lambda(z,v_z)t/2}$ for any $t>0$ sufficiently large. Therefore, $\omega(z)$ only consists of equilibria, which implies that $\omega(z)\subset \omega_0(x)$. Based on our assumption, we have $\lambda_{k\tilde{z}}>0$ for any $\tilde{z}\in \omega(z)$.  By virtue of Theorem $\ref{omega-lambdak>0}$, we obtain that $z\in \overline{Q}$.

Since $z\in\omega(x)$, we choose a sequence $t_n\to \infty$ such that $\Phi_{t_n}(x)\to z$ as $n\to \infty$. Now we define $$C_x=\{y\in D: y\neq x\,\text{ and }\,y\thicksim x\}.$$
Clearly, $C_x$ is nonempty, since $x\in D$, $D$ is open and $C$ is $k$-solid. Moreover, $C_x$ is $\omega$-compact because $D$ is $\omega$-compact.
So, if there exists some $y\in C_{x}$ with a subsequence of $\{t_n\}_{n=1}^{\infty}$, still denoted by $\{t_n\}_{n=1}^{\infty}$, such that $\Phi_{t_n}(y)\to z$, then by Lemma \ref{co-limit}, one has either $z$ is an equilibrium or $z\in Q$. Recall that $z\notin \omega_0(x)$. So $z\in Q$. Hence, we again obtain $x\in Q$ and we have done.

On the other hand, if for any $y\in C_{x}$, there is a subsequence $t^y_{n_k}\to \infty$ of $\{t_n\}_{n=1}^{\infty}$ such that $\Phi_{t^y_{n_k}}(y)\to z_y\neq z$ as $k\rightarrow \infty$, then it is clear that $z_y\thicksim z$ for any $y\in C_x$. Since $\lambda_{k\tilde{z}}>0$ for any $\tilde{z}\in \omega(z)$, Lemma  \ref{Distance-nonlinear-system} (for $\omega(z)$) implies that there is a $\delta>0$ such that $$\limsup_{t\to +\infty}\norm{\Phi_t(z_y)-\Phi_t(z)}\ge \delta$$ for any $y\in C_x$. As a consequence, for each $y\in C_x$ (hence, for each $z_y$), we can choose without loss of generality a sequence $s^y_n\to \infty$ such that

(P1). $\Phi_{s^y_n}(z_y)\rightarrow z_{z_y}$ and $\Phi_{s^y_{n}}(z)\rightarrow z_z^y$ as $n\rightarrow +\infty$; and

(P2). $z_z^y\thicksim z_{z_y}$ with $\norm{z_z^y-z_{z_y}}\geq \delta$. \vskip 1mm

\noindent Based on this, we {\it further claim that}, for each $y\in C_x$, there is a sequence $\{\tau^y_n\}_{n=1}^{\infty}$ such that $\Phi_{\tau_{n}^y}(x)\rightarrow z_z^{y}\in \omega(x)$ and $\Phi_{\tau_{n}^y}(y)\rightarrow z_{z_y}\in \omega(y)$ as $n\rightarrow +\infty$ satisfying  $z_z^y\thicksim z_{z_y}$ and $\norm{z_z^y-z_{z_y}}\geq \delta$. Before proving this claim, we first show how it implies $x\in \overline{Q}$. Take any a sequence $\{x_k\}_{k=1}^{\infty}\subset C_x$ with $x_k\to x$. For each $x_k$, we utilize the claim to obtain $z_z^{x_k}\in \omega(x)$ and $z_{z_{x_k}}\in \omega(x_k)$ satisfying $z_z^{x_k}\thicksim z_{z_{x_k}}$ and $\norm{z_z^{x_k}-z_{z_{x_k}}}\geq \delta$. Then, due to the $\omega$-compactness of $C_x$, one can repeat the argument in the last two paragraphs of the proof of Theorem $\ref{omega-lambdak>0}$ to obtain that there is a subsequence $\{k_l\}_{l=1}^\infty$ such $x_{k_l}\in Q$ for all $l$ sufficiently large. Consequently, we have $x\in \overline{Q}$, which completes the proof.

Finally, it remains to prove the claim. Due to (P2), it is clear that $z_z^y\thicksim z_{z_y}$ and $\norm{z_z^y-z_{z_y}}\geq \delta$. So, we only need to show the existence of the sequence $\{\tau^y_n\}_{n=1}^{\infty}$. To this end, we observe (P1), that is, $\Phi_{s_{n}^y}(z_y)\rightarrow z_{z_y}$ and $\Phi_{s_{n}^y}(z)\rightarrow z_z^y$ as $n\rightarrow +\infty$. So, one can find a subsequence $\{N_{1}(n)\}_{n=1}^\infty$ of positive integers such that $$\norm{\Phi_{s_{N_1(n)}^y}(z)-z_{z}^y}<\frac{1}{2n}\,\text{ and }\,
\norm{\Phi_{s_{N_1(n)}^y}(z_y)-z_{z_y}}<\frac{1}{2n}$$ for every $n\ge 1$. Recall also that $\Phi_{t^y_{n_k}}(x)\rightarrow z$ and $\Phi_{t^y_{n_k}}(y)\rightarrow z_y$ as $k \rightarrow+\infty$. Then, for each $n$, one can further choose an integer $N_2(n)> 1$ such that $$ \begin{aligned} \norm{\Phi_{s_{N_1(n)}^y}(\Phi_{t_{N_2(n)}^y}(x))-\Phi_{s_{N_1(n)}^y}(z)}<\frac{1}{2n},\\\norm{\Phi_{s_{N_1(n)}^y}(\Phi_{t_{N_2(n)}^y}(y))-\Phi_{s_{N_1(n)}^y}(z_y)}<\frac{1}{2n}\end{aligned}$$ for every $n\ge 1$. Hence, we have $$\begin{aligned}\norm{\Phi_{s_{N_1(n)}^y+t_{N_2(n)}^y}(x)-z_z^y}<\frac{1}{n}
\,\text{ and }\,\norm{\Phi_{s_{N_1(n)}^y+t_{N_2(n)}^y}(y)-z_{z_y}}<\frac{1}{n} \end{aligned}$$ for every $n\ge 1$. Let $\tau^y_{n}=s_{N_1(n)}^y+t_{N_2(n)}^y$ for $n\ge 1$. This establishes the existence of $\{\tau^y_{n}\}$ and the claim verifies.
\end{proof}

\section{Generic behavior and Poincar\'{e}-Bendixson Theorem}

Based on our discussion in the previous sections, we can now describe in this section the {\it generic behaviors} of the flow $\Phi_t$ strongly monotone flow with respect to $k$-cone $C$ (see Theorem \ref{T:Open-dense} or Theorem A), which concludes that generic (open and dense) orbits are
either pseudo-ordered or convergent to equilibria.

In particular, when $k=2$, together with the results obtained in \cite{F-W-W}, we will further show that the {\it generic} orbit of $\Phi_t$ satisfies the Poincar\'{e}-Bendixson Theorem (see Theorem \ref{T:Open-dense-PB} or Theorem B), that is, for generic (open and dense) points the $\omega$-limit set containing no equilibria is a single closed orbit.
This result will be referred as the {\it generic Poincar\'{e}-Bendixson Theorem} for $\Phi_t$.

Before we state our main theorems, we need more notations. We denote $C_E$ as
$$C_E=\{x\in \mathbb{R}^n: \text{the orbit }O(x) \text{ converges to equilibrium}\}.$$
For any $D\subset \RR^n$, we recall that the orbit set of $D$ is defined as $\mathcal{O}(D)=\bigcup_{x\in D}O(x).$

\begin{theorem}\label{T:Open-dense}
Assume that {\rm (FWW)} hold. Let $\mathcal{D}\subset \RR^n$ be an open bounded set such that the orbit set $\mathcal{O}(\mathcal{D})$ of $\mathcal{D}$ is bounded. Then ${\rm Int}(Q\cup C_E)$ {\rm ({\it interior in $\RR^n$})} is dense in $\mathcal{D}$.
\end{theorem}
\begin{proof}
Given any $\bar{x}\in \mathcal{D}$ and any neighborhood $U$ of $\bar{x}$ in $\mathcal{D}$. If $U\subset Q\cup C_E$, then one has $\bar{x}\in {\rm Int}(Q\cup C_E)$. Thus, we are done.
So, we assume that there exists some $x\in U\setminus{(Q\cup C_E)}$. Before going further, we note that $\mathcal{D}$  is $\omega$-compact because $\mathcal{O}(\mathcal{D})$ is bounded. Thus, all the results in Section 4 hold for such $x$. So, one of following three alternatives must occur:

(a) $\lambda_{kz}\le 0$ for some point $z\in \omega_0(x)$;

(b) $\lambda_{kz}>0$ for any $z\in \omega_0(x)$, and there exists a $\tilde{z}\in\omega(x)\setminus \omega_0(x)$ such that $\lambda_{k\tilde{z}}\leq 0$;

(c) $\lambda_{kz}>0$ for any $z\in \omega(x).$

\noindent Here, $\omega_0(x)$ is the set of regular points in $\omega(x)$ defined
in \eqref{E:rg-point}.
Since $x\notin Q\cup C_E$, Theorem \ref{omega-lambdak<=0} directly yields that Case (a) cannot happen. For Case (c), it follows from Theorem  \ref{omega-lambdak>0} that $x\in \bar{Q}$. So, one can choose some $y\in Q$ so close to $x$ that $y\in U$; and moreover, since $Q$ is open, we have $y\in Q={\rm Int}Q\subset {\rm Int}(Q\cup C_E)$. For case (b), we can deduce from Theorem \ref{regular-lambdak>0} that $x\in \bar{Q}$, which directly implies that there exists some $y$ in $U$ satisfying $y\in {\rm Int}(Q\cup C_E)$.

By arbitrariness of $\bar{x}$ and $U$, we have proved that ${\rm Int}(Q\cup C_E)$ is dense in $X$.
\end{proof}

\begin{rem}
{\rm Theorem \ref{T:Open-dense} states that, for smooth flow $\Phi_t$ strongly monotone with respect to $k$-cone $C$, generic (open and dense) orbits are
either pseudo-ordered or convergent to equilibria.
If the rank $k=1$, Theorem \ref{T:Open-dense} automatically implies Hirsch's Generic Convergence Theorem on $\RR^n$ due to the Monotone Convergence Criterion.
}
\end{rem}

\vskip 3mm
Now we state the the {\it generic Poincar\'{e}-Bendixson Theorem} for the $\Phi_t$.
\begin{theorem}\label{T:Open-dense-PB}
Assume that {\rm (FWW)} hold and $k=2$. Let $\mathcal{D}\subset \RR^n$ be an open bounded set such that the orbit set $\mathcal{O}(\mathcal{D})$ of $\mathcal{D}$ is bounded.  Then, for generic (open and dense) points $x\in D$, the $\omega$-limit set $\omega(x)$ containing no equilibria is a single closed orbit.
\end{theorem}
\begin{proof}
 Let $\mathcal{G}={\rm Int}(Q\cup C_E)\subset X$. By Theorem \ref{T:Open-dense}, $\mathcal{G}$ is open and dense in $X$. Now, given any $x\in \mathcal{G}$, if $\omega(x)\cap E=\emptyset$ then one has $x\in Q$. Consequently, it follows from \cite[Theorem 5.1]{F-W-W} that $\omega(x)$ consists of a periodic point, which completes the proof.
\end{proof}

\section{Applications to high-dimensional systems}
In this section, we demonstrate our general results by establishing the so-called  {\it generic Poincar\'{e}-Bendixson Theorem} for high-dimensional autonomous systems of ordinary differential equations. This means that for a {\it generic (open and dense)} initial point in the phase space, the omega-limit set containing no equilibria must be a single closed orbit.

\vskip 3mm
We consider a general autonomous system of ODEs
\begin{equation}\label{E:ODE-sys}
\dot{x}=F(x),\quad x\in \mathbb{R}^n,
\end{equation}
in which $F$ is a $C^{1,\alpha}$-smooth vector field defined in $\RR^n$. We denote by $\Phi_t(x)$ the flow generated by \eqref{E:ODE-sys}. System \eqref{E:ODE-sys} is called dissipative if there is an open bounded set $\mathcal{B}\subset\RR^n$ such that for each $x\in \RR^n$ there is a $t_0>0$ such that $\Phi_t(x)\in \mathcal{B}$ for all $t\ge t_0$.

Let $C\subset \mathbb{R}^n$ be a $2$-solid cone which is complemented. System \eqref{E:ODE-sys} is said to be
{\it $C$-cooperative} if the fundamental solution matrix $U^{pq}(t)$ of the linear system
$$\dot{U}=A^{pq}(t)U, \quad U(0)=I$$ satisfies the cone invariance condition
\begin{equation}\label{E:cone-condition-ODE}
U^{pq}(t)(C\setminus \{0\})\subset {\rm Int}C,\,\,\,\text{ for all }t>0.
\end{equation}
Here, the associated matrix $A^{pq}(t)$ with $p,q\in \RR^n$ is
$$A^{pq}(t)=\int_0^1DF(s\Phi_t(p)+(1-s)\Phi_t(q))ds.$$
\vskip 2mm

Now, we give the following Poincar\'{e}-Bendixson Theorem for system \eqref{E:ODE-sys}.

\begin{theorem}\label{T:ODE-poincare}
Assume that system \eqref{E:ODE-sys} is $C^{1,\alpha}$-smooth, dissipative and $C$-cooperative. Then there is an open and dense subset $\mathcal{D}\subset \RR^n$ such that for any $x\in \mathcal{D}$, if the $\omega$-limit set $\omega(x)$ contains no equilibrium then $\omega(x)$ is a periodic orbit.
\end{theorem}
\begin{proof}
Together with \cite[Proposition 1]{San09} by Sanchez, it is clear that if system \eqref{E:ODE-sys} is $C^{1,\alpha}$-smooth and $C$-cooperative, then the flow $\Phi_t(x)$ satisfies the assumption (FWW).
Take any integer $i\ge 1$ and let $B_i$ be the open ball centered at the origin with radius $i$.
Since system \eqref{E:ODE-sys} is dissipative, the orbit set $\mathcal{O}(B_i)$ of $B_i$ is bounded. By Theorem \ref{T:Open-dense-PB}, there is an open and dense set $D_i$ in $B_i$ from which the omega-limit set containing no equilibria is a closed orbit. Note that $\RR^n=\cup_{i\ge 1}B_i$. Then $\mathcal{D}:=\cup_{i\ge 1}D_i$ is an open and sense subset in $\RR^n$. Moreover, for any $x\in \mathcal{D}$, if the $\omega$-limit set $\omega(x)$ contains no equilibrium then $\omega(x)$ is a periodic orbit. We have completed the proof.
\end{proof}

\begin{rem}
{\rm Theorem \ref{T:ODE-poincare} will be referred to as the {\it generic Poincar\'{e}-Bendixson Theorem} for high-dimensional ODE systems. Based on this Theorem, we improve that the Poincar\'{e}-Bendixson type conclusion in \cite{San09,F-W-W} is  satisfied for generic orbits, instead of just certain (i.e., pseudo-ordered) orbits.
}
\end{rem}

\vskip 2mm

Finally, we specify a quadratic cone which is $2$-solid and complemented.
Let $P$ be a constant real symmetric non-singular matrix $n\times n$ matrix, with $2$ negative eigenvalues and $(n-2)$ positive eigenvalues. Then the set
\begin{equation*}\label{E:2-cone}
C^-(P)=\{x\in \mathbb{R}^n:x^*Px\le 0\}
\end{equation*} is a $2$-solid cone which is also complemented. Here $x^*$ denote the transpose of the vector $x\in \mathbb{R}^n$.

Assume that there exists a continuous function $\lambda:\mathbb{R}^n\to \mathbb{R}$ (not necessarily positive) such that the matrices \begin{equation}\label{decay-11-smooth-OSan1}
PDF(x)+(DF(x))^*P+\lambda(x) P<0,\,\, \text{ for any }x\in \mathbb{R}^n,
\end{equation}
where $DF(x)^*$ stands for the transpose of the Jacobian $DF(x)$ and $``<"$ represents the usual order in the space of symmetric matrices (i.e., the matrices are negative definite). The following lemma is due to Sanchez \cite{San09,San10}.
\begin{lemma}\label{L:San}
Assume that \eqref{decay-11-smooth-OSan1} holds. Then system \eqref{E:ODE-sys} is $C^-(P)$-cooperative.
\end{lemma}
\begin{proof}
See Sanchez \cite[Proposition 7]{San09} and his afterwards discussion in \cite{San09}.
\end{proof}

Combing with Theorem \ref{T:ODE-poincare} and Lemma \ref{L:San}, we obtain the following generic Poincar\'{e}-Bendixson Theorem for high-dimensional system \eqref{E:ODE-sys} with a quadratic cone:
\begin{cor}\label{C:quadratic-poincare}
Assume that system \eqref{E:ODE-sys} is $C^{1,\alpha}$-smooth, dissipative and satisfies \eqref{decay-11-smooth-OSan1}. Then the conclusion of Theorem \ref{T:ODE-poincare} holds.
\end{cor}

\begin{rem}
\textnormal{In \cite{SmithR-2,SmithR-3}, R.A. Smith succeeded in establishing a Poincar\'{e}-Bendixson theorem for the high-dimensional system \eqref{E:ODE-sys}
by assuming that $F$ satisfies
\begin{equation}\label{decay-12-smith}
(x-y)^*\cdot P\cdot [F(x)-F(y)+\lambda (x-y)]\le -\epsilon\abs{x-y}^2
\end{equation}
for any $x,y\in \RR^n$, where $\lambda,\epsilon>0$ are positive constants and $\abs{x-y}$ denote the Euclidean norm of the vector $x-y$. It states that {\it any omega-limit set} containing no equilibria must be a single closed orbit.
If $F$ is of class $C^1$, by following the same arguments in Ortega and Sanchez \cite[Remarks 1-2]{OSan}, one may obtain that
\eqref{decay-12-smith} holds if and only if
\begin{equation}\label{decay-13-smooth}
PDF(x)+(DF(x))^*P+\lambda P\le -\epsilon I,\,\, \text{ for any }x\in \RR^n.
\end{equation}
In his proof \cite{SmithR-2,SmithR-3}, the matrix $P$ is employed to produce a quadratic Lyapunov function that helps to establish his Poincar\'{e}-Bendixson theorem.}

\textnormal{
Compared to \eqref{decay-11-smooth-OSan1}, the assumption \eqref{decay-13-smooth} in R. A. Smith's work requires the matrices be negative definite in a uniform sense with respect to $x$, and $\lambda$ is a positive constant. Therefore, under the weaker assumption \eqref{decay-11-smooth-OSan1}, the Lyapunov-function approach in \cite{SmithR-2,SmithR-3} does not work any more. However, our Corollary \ref{C:quadratic-poincare} concludes that, {\it generically, the Poincar\'{e}-Bendixson theorem still holds}.
}
\end{rem}


\begin{thebibliography}{99}

\bibitem{AnS} D. Angeli and E. D. Sontag, Monotone control systems, IEEE Trans. Automat. Control, 48
(2003), 1684-1698.

\bibitem{Ba} S. Baigent, Geometry of carrying simplices of 3-species competitive Lotka-Volterra systems, Nonlinearity, 26(2013), 1001-1029.

\bibitem{BG} J. Bochi and N. Gourmelon, Some characterization of domination, Math. Z., 263 (2009), 221-231.


\bibitem{BDV} C. Bonatti, L. J. D\'{\i}az and M. Viana, Dynamics beyond uniform hyperbolicity. A global geometric and probabilistic perspective. Encyclopaedia of Mathematical Sciences, 102. Mathematical Physics, III. Springer-Verlag, Berlin, 2005.

\bibitem{CK} F. Colonius and W. Kliemann, The dynamics of control. Systems $\&$ Control: Foundations $\&$ Applications. Birkh$\ddot{a}$user Boston, Inc., Boston, MA, 2000.

\bibitem{FGW13} C. Fang, M. Gyllenberg and Y. Wang, Floquet bundles for tridiagonal competitive-cooperative systems and the dynamics of time-recurrent systems, SIAM J. Math. Anal., 45 (2013), 2477-2498.

\bibitem{Fied89} B. Fiedler, Discrete Ljapunov functionals and $\omega$-limit sets, RAIRO Mod\'{e}l. Math. Anal. Num\'{e}r., 23(1989), 415-431.

\bibitem{FM89} B. Fiedler and J. Mallet-Paret, A Poincar\'{e}-Bendixson theorem for scalar reaction diffusion equations, Arch. Ration. Mech. Anal., 107 (1989) 325-345.



\bibitem{Fusco87} G. Fusco and W. Oliva, Jacobi matrices and transversality, Proc. Roy. Soc. Edinburgh Sect. A, 109 (1988), 231-243.

\bibitem{FO1} G. Fusco and W. Oliva, A Perron theorem for the existence of invariant subspaces. Annali di Matematica Pura ed Applicata, 160 (1991), 63-76.

\bibitem{F-W-W} L. Feng, Y. Wang and J. Wu, Semiflows ``monotone with respect to high-rank cones" on a Banach space, SIAM J. Math. Anal., 49(2017), 142-161.

\bibitem{Gan} F. Gantmacher, The Theory of Matrices, vol. 2, Chelsea Publ. Co., New York, 1959.

\bibitem{Ge} T. Gedeon, Cyclic feedback systems, Mem. Amer. Math. Soc., 134(1998), no. 637.

\bibitem{Hir0} M. W. Hirsch, Stability and convergence in strongly monotone dynamical systems, Journal f\"{u}r die reine und angewandte Mathematik (Crelles Journal), 383 (1998), 1-53.

\bibitem{Hir1} M. W. Hirsch, Systems of differential equations which are competitive or cooperative. I. Limit sets, SIAM J. Math. Anal., 13(1982), 167-179.

\bibitem{Hir2} M. W. Hirsch, Systems of differential equations that are competitive or cooperative. II: Convergence almost everywhere, SIAM J. Math. Anal., 16 (1985), 423-439.

\bibitem{Hir3} M. W. Hirsch, Systems of differential equations which are competitive or cooperative. III. Competing species, Nonlinearity, 1 (1988), 51-71.

\bibitem{Hir4} M. W. Hirsch, Systems of differential equations that are competitive or cooperative. IV: Structural stability in three-dimensional systems, SIAM J. Math. Anal., 21 (1990), 1225-1234.

\bibitem{Hir5} M. W. Hirsch, Systems of differential equations that are competitive or cooperative. V. Convergence in 3-dimensional systems, J. Differential Equations, 80 (1989), 94-106.

\bibitem{Hir6} M. W. Hirsch, Systems of differential equations that are competitive or cooperative. VI. A local $C^r$ closing lemma for 3-dimensional systems, Ergodic Theory Dynam. Systems, 11 (1991),  443-454.

\bibitem{Hir-Smi} M. W. Hirsch and H. Smith, Monotone dynamical systems, in Handbook of Differential Systems (Ordinary Differential Equations), vol. 2, Elsevier, Amsterdam, 2005, 239-358.

\bibitem{HP} P. Hess and P.Pol\'{a}\v{c}ik, Boundedness of prime periods of stable cycles and convergence to fixed points in discrete monotone dynamical systems. SIAM J. Math. Anal., 24 (1993), 1312-1330.

\bibitem{HuPo} J. H\'{u}ska and P. Pol\'{a}\v{c}ik, The principal Floquet bundle and exponential separation for linear parabolic equations, J. Dynamics Differential Equations, 16 (2004), 347-375.

\bibitem{John-Pa-S} R. A. Johnson, K. J. Palmer and G. R. Sell, Ergodic properties of linear dynamical systems, SIAM J. Math. Anal., 18 (1987), 1-33.
\bibitem{JMW} J. Jiang, J. Mierczy\'{n}ski and Y. Wang, Smoothness of the carrying simplex for discrete-time competitive dynamical systems: a characterization of neat embedding. J. Differential Equations 246 (2009), 1623-1672.

\bibitem{JR} R. Joly and G. Raugel, Generic Morse-Smale property for the parabolic equation on the circle, Ann. Inst. H. Poincar¡äe, Anal. Non Lin¡äeaire, 27 (2010), 1397-1440.

\bibitem{Kato} T. Kato, Pertubation Theory for Linear Operators, Springer, New York, 1980.

\bibitem{KLS} M. A. Krasnosel'skij, J. A. Lifshits, and A. V. Sobolev, Positive Linear Systems, the Method of Positive Operators, Heldermann Verlag, Berlin, 1989.

\bibitem{KR} M.G. Krein and M.A. Rutman, Linear operators leaving invariant a cone in a Banach space (in Russian), Uspekhi Mat. Nauk (N.S.), 31(23) (1948) 3-95, English translation in Amer. Math. Soc. Transl., 26 (1950).

\bibitem{LL} Z. Lian and K. Lu, Lyapunov Exponents and Invariant Manifolds for Random Dynamical Systems on a Banach Space, Mem. Amer. Math. Soc. 206 (2010), no. 967.

\bibitem{LW1} Z. Lian and Y. Wang, \newblock $K$-dimensional invariant cones of random dynamical system in $\mathbb{R}^n$ with applications, J. Differential Equations, 259 (2015), 2807-2832.

\bibitem{LW2} Z. Lian and Y. Wang, On Linear random dynamical systems in a Banach space: I. Multiplicative Ergodic Theorem and Krein-Rutmann type Theorems, Adv. Math., 312(2017), 374-424.


\bibitem{M-PN1} J. Mallet-Paret and R.D. Nussbaum, Generalizing the Krein-Rutman theorem,measures of noncompactness and the fixed point index, J. Fixed Point Theory Appl. 7 (2010), 103-143.


\bibitem{M-PN} J. Mallet-Paret and R.D. Nussbaum, Tensor products, positive linear operators, and delay-differential equations, J. Dynamics Differential Equations, 25 (2013), 843-905.

\bibitem{M-PSmi90} J. Mallet-Paret and H. Smith, The Poincare-Bendixson theorem for monotone cyclic feedback systems, J. Dynamics Differential Equations, 2 (1990), 367--421.

\bibitem{M-PSe96} J. Mallet-Paret and G. Sell, Systems of differential delay equations: Floquet multipliers and discrete lyapunov functions, J. Differential Equations, 125 (1996), 385--440.

\bibitem{M-PSe96-2} J. Mallet-Paret and G. Sell, The Poincar\'{e}-Bendixson theorem for monotone cyclic feedback systems with delay, J. Differential Equations, 125 (1996), 441-489.

\bibitem{MSon} M. Margaliot and E. D. Sontag, Revisiting totally positive differential systems: A tutorial and new results, Automatica, 101 (2019), 1-14.

\bibitem{Mata1} H. Matano, Asymptotic behavior and stability of solutions of semilinear diffusion equations,
Publ. RIMS Kyoto Univ., 15 (1979), 401-454.

\bibitem{Mata2} H. Matano, Strongly order-preserving local semi-dynamical systems-theory and applications,
Semigroups, Theory and Applications, vol.1, H.Brezis, M.G.Crandall, and F.Kappel, eds.,
Res. Notes in Math., Longman Scientific and Technical, London, 141 (1986), 178-185.

\bibitem{Mi1}  J. Mierczy\'{n}ski,  The $C^1$ property of convex carrying simplices for competitive maps, Ergodic Theory Dynam. Systems, in press.

\bibitem{Mi11} J. Mierczy\'{n}ski,  The $C^1$ property of convex carrying simplices for a class of competitive systems of ODEs, J. Differential Equations, 111 (1994), 385-409.

\bibitem{Mi2}  J. Mierczy\'{n}ski,  Flows on ordered bundles, preprint, 1995.

\bibitem{MiSh1} J. Mierczy\'{n}ski and W. Shen, Principal Lyapunov exponents and principal Floquet spaces of positive random dynamical systems. I. General theory, Trans. Amer. Math. Soc., 365 (2013), 5329-5365.

\bibitem{MiSh2} J. Mierczy\'{n}ski and W. Shen, Principal Lyapunov exponents and principal Floquet spaces of positive random dynamical systems. II. Finite-dimensional systems, J. Math. Anal. Appl., 404 (2013), 438-458.

\bibitem{MiSh3} J. Mierczy\'{n}ski and W. Shen, Principal Lyapunov exponents and principal Floquet spaces of positive random dynamical systems. III. Parabolic Equations and Delay Systems, J Dynamics Differential Equations. 28 (2016), 1039-1079.

\bibitem{MiSh4}  J. Mierczy\'{n}ski and W. Shen, Spectral Theory for Random and Nonautonomous Parabolic Equations and Applications, Chapman \& Hall/CRC Monographs and Surveys in Pure and Applied Mathematics, 139, Chapman, 2008.

\bibitem{Nu} R. D. Nussbaum, Eigenvectors of nonlinear positive operators and thelinear Krein-Rutman theorem, Fixed Point Theory (Sherbrooke, Que., 1980), Lecture Notes in Math., vol. 886, 309-330. Springer, Berlin, 1981.

\bibitem{OSan} R. Ortega and L. A. Sanchez, Abstract competitive systems and orbital stability in $\mathbb{R}^3$, Proc. Amer., Math. Soc., 128 (2000), 2911-2919.

 \bibitem{OSe} V. I. Oseledets, A multiplicative ergodic theorem. Lyapunov characteristic numbers for dynamical systems. Trans. Moscow Math. Soc., 19 (1968), 197-31 .

 \bibitem{QTZ} A. Quas, P. Thieullen and M. Zarrabib, Explicit bounds for separation between Oseledets subspaces, Dynamical Systems (2019), in press.

\bibitem{Ru} D. Ruelle, Analycity properties of the characteristic exponents of random matrix products, Adv. in Math., 32 (1979), 68-80.

\bibitem{Pola1} P. Pol\'{a}\v{c}ik, Convergence in smooth strongly monotone
ows defined by semilinear parabolic equations, J. Differential Equations., 79 (1989), 89-110.

\bibitem{Pola2} P. Pol\'{a}\v{c}ik, Generic properties of strongly monotone semifows defined
by ordinary and parabolic differential equations, Qualitative theory of diferential equations (Szeged, 1988), 519-530, Colloq. Math. Soc. J\'{a}nos Bolyai, 53, North-Holland, Amsterdam, 1990.

\bibitem{Po1} P. Pol\'{a}\v{c}ik, Parabolic equations: asymptotic behavior and dynamics on
invariant manifolds, Handbook on Dynamical Systems, vol. 2, Elsevier, Amsterdam, 2002, 835-883.


\bibitem{PT1} Pol\'{a}\v{c}ik and I. Tere\v{s}\v{c}\'{a}k, Convergence to cycles as a typical asymptotic behavior in smooth strongly monotone discrete-time dynamical systems, Arch. Ration. Mech. Anal., 116 (1992), 339-360.

\bibitem{PT2} Pol\'{a}\v{c}ik and I. Tere\v{s}\v{c}\'{a}k,  Exponential separation and invariant bundles for maps in ordered Banach spaces with applications to parabolic equations, J. Dynamics Differential Equations, 5 (1993), 279-303.

\bibitem{Puj} E. Pujals, From hyperbolicity to dominated splitting. Partially hyperbolic dynamics, laminations, and Teichm\"{u}ller flow, 89-102, Fields Inst. Commun., 51, Amer. Math. Soc., Providence, RI, 2007.

\bibitem{San09} L. A. Sanchez, Cones of rank 2 and the Poincar\'{e}-Bendixson property for a new class of monotone systems, J. Differential Equations, 216 (2009), 1170-1190.

\bibitem{San10} L. A. Sanchez, Existence of periodic orbits for high-dimensional autonomous systems, J. Math. Anal. Appl., 363 (2010), 409-418.

\bibitem{SanF} B. Sandstede and B. Fiedler, Dynamics of periodically forced parabolic equations on the circle, Ergodic Theory Dynam. Systems., 12 (1992) 559-571.


\bibitem{Smi-Thi1} H. Smith and H. Thieme, Quasi convergence and stability for strongly order-preserving semiflows, SIAM J. Math. Anal., 21 (1990), 673-692.

\bibitem{Smillie} J. Smillie, Competitive and cooperative tridiagonal systems of differential equations, SIAM J. Math. Anal., 15 (1984), 530--534.

\bibitem{Smith91} H. Smith, Periodic tridiagonal competitive and cooperative systems of differential equations, SIAM J. Math. Anal., 22 (1991), 1102--1109.

\bibitem{Smi95} H. Smith, Monotone Dynamical Systems, an Introduction to the Theory of Competitive and Cooperative Systems, Mathematical Surveys
and Monographs, Vol. 41, American Mathematical Society, Providence, RI, 1995.

\bibitem{Smi17} H. Smith, Monotone dynamical systems: reflections on new advances and applications. Discrete Contin. Dyn. Syst., 37 (2017), No. 1, 485-504.

\bibitem{SmithR-2} R. A. Smith, Existence of periodic orbits of autonomous ordinary differential equations, Proc. of Royal Soc. of Edinburgh A, 85 (1980), 153-172.

\bibitem{SmithR-3} R. A. Smith, Orbital stability for ordinary differential equations, J. Differential Equations, 69 (1987), 265-287.


\bibitem{SWZ} W. Shen, Y. Wang and D. Zhou, Long-time behavior of almost periodically forced parabolic equations on the circle, J. Differential Equations, 266 (2019), 1377-1413.

\bibitem{Te1} I. Tere\v{s}\v{c}\'{a}k, Dynamical systems with discrete Lyapunov functionals, Ph.D. thesis, Comenius University, Bratislava, 1994.

\bibitem{Te2} I. Tere\v{s}\v{c}\'{a}k, Dynamics of $C^1$ smooth strongly monotone discrete-time dyanmical systems, preprint, Comenius University, Bratislava, 1994.

\bibitem{Vi} M. Viana, Lectures on Lyapunov Exponents, Cambridge Studies in Advanced Mathematics (145), Cambridge University Press, 2014.

\bibitem{WJ} Y. Wang and J. Jiang, The general properties of discrete-time competitive dynamical systems, J. Differential Equations, 176 (2001), 470-493.


\end{thebibliography}
\end{document}